\documentclass[a4paper, 12pt]{article}

\usepackage[utf8]{inputenc}
\usepackage{lmodern}
\usepackage{geometry}
\usepackage{amssymb, amsfonts, amsmath, amsthm,  mathtools}
\usepackage[english]{babel}
\usepackage[colorlinks]{hyperref}
\usepackage{tikz}
\usepackage{bbm}

\theoremstyle{plain}
\newtheorem{theorem}{Theorem}[section]
\newtheorem{corollary}[theorem]{Corollary}
\newtheorem{lemma}[theorem]{Lemma}
\newtheorem{proposition}[theorem]{Proposition}

\theoremstyle{definition}

\theoremstyle{remark}
\newtheorem{remark}[theorem]{Remark}

\newcommand{\N}{\mathbb{N}}
\newcommand{\Z}{\mathbb{Z}}
\newcommand{\R}{\mathbb{R}}

\newcommand{\floor}[1]{{\left\lfloor #1 \right\rfloor}}

\DeclareMathOperator{\1}{\mathbbm{1}}

\numberwithin{equation}{section}

\DeclareMathOperator{\E}{\mathbf{E}}
\newcommand{\rmE}{\mathbb{E}}
\renewcommand{\P}{\mathbf{P}}
\newcommand{\rmP}{\mathbb{P}}
\DeclareMathOperator{\Var}{\mathbf{V}\mathrm{ar}}
\DeclareMathOperator{\rmVar}{\mathbb{V}\mathrm{ar}}
\DeclareMathOperator{\rmCov}{\mathbb{C}\mathrm{ov}}

\newcommand{\calL}{\mathcal{L}}

\newcommand{\calK}{\mathcal{K}}

\newcommand{\dd}{\mathrm{d}}

\renewcommand{\bar}[1]{\overline{#1}}

\renewcommand{\hat}[1]{\widehat{#1}}

\renewcommand{\rho}{\varrho}
\renewcommand{\epsilon}{\varepsilon}

\title{Late levels of nested occupancy scheme in random environment}
\author{Alexander Iksanov\thanks{\texttt{iksan@univ.kiev.ua}, Taras Shevchenko National University of Kyiv, Kyiv-01033, Ukraine} \and Bastien Mallein\thanks{\texttt{mallein@math.univ-paris13.fr}, Université Sorbonne Paris Nord, LAGA, UMR 7539, F-93430, Villetaneuse, France}}
\date{\today}

\begin{document}

\maketitle

\begin{abstract}
Consider a weighted branching process generated by a point process on $[0,1]$, whose atoms sum up to one. Then the weights of all individuals in any given generation sum up to one, as well. We define a nested occupancy scheme in random environment as the sequence of balls-in-boxes schemes (with random probabilities) in which boxes of the $j$th level, $j=1,2,\ldots$ are identified with the $j$th generation individuals and the hitting probabilities of boxes are identified with the corresponding weights. The collection of balls is the same for all generations, and each ball starts at the root and moves along the tree of the weighted branching process according to the following rule: transition from a mother box to a daughter box occurs with probability given by the ratio of the daughter and mother weights.

Assuming that there are $n$ balls, we give a full classification of regimes of the a.s.\ convergence for the number of occupied (ever hit) boxes in the $j$th level, properly normalized, as $n$ and $j=j_n$ grow to $\infty$. Here, $(j_n)_{n\in\N}$ is a sequence of positive numbers growing proportionally to $\log n$. We call such levels late, for the nested occupancy scheme gets extinct in the levels which grow super-logarithmically in $n$ in the sense that each occupied box contains one ball. Also, in some regimes we prove the strong laws of large numbers (a) for the number of the $j$th level boxes which contain at least $k$ balls, $k\geq 2$ and (b) under the assumption that the mean number of the first level boxes is finite, for the number of empty boxes in the $j$th level.
\end{abstract}

\section{Introduction}\label{sec:intro}

\subsection{Definition of the model}\label{subsec:definition}

Let $P_1$, $P_2,\ldots$ be nonnegative random variables with an arbitrary joint distribution satisfying $\sum_{k\geq 1}P_k=1$ almost surely (a.s.). We interpret the sequence $(P_k)_{k\in\N}$ as a random environment. In the {\it occupancy scheme in random environment}, conditionally on
$(P_k)_{k\in\N}$, balls are thrown independently into an array of boxes $1$, $2,\ldots$ with probability $P_k$ of hitting box $k$. To exclude the trivial occupancy scheme with one ball, we always assume that the mean number of positive (nonzero) probabilities is larger than one.

We call the model {\it infinite occupancy scheme in random environment} provided that the number of positive probabilities is infinite a.s. When the random probabilities $(P_k)_{k\in\N}$ follow a residual allocation model
\begin{equation}\label{BS}
P_r:=W_1W_2\cdot\ldots\cdot W_{r-1}(1-W_r),\quad r\in\N,
\end{equation}
where $W_1$, $W_2,\ldots$ are independent copies of a random variable $W$ taking values in $(0,1)$, the corresponding infinite occupancy scheme is called {\it Bernoulli sieve}. This scheme was introduced in \cite{Gnedin:2004} and further investigated in many articles. See \cite{Gnedin+Iksanov+Marynych:2010} and \cite{Iksanov:2016} for the surveys and \cite{Alsmeyer+Iksanov+Marynych:2017, Duchamps+Pitman+Tang:2019, Iksanov+Jedidi+Bouzeffour:2017, Pitman+Tang:2019} for recent contributions. A list of other popular random environments can be found in \cite{Gnedin+Iksanov:2020}.

A {\it nested occupancy scheme in random environment} is a sequence of occupancy schemes in random environments constructed as a successive refinement of the partition of balls. The model was introduced in \cite{Bertoin:2008}. The other contributions that we are aware of are \cite{Buraczewski+Dovgay+Iksanov:2020, Gnedin+Iksanov:2020, Iksanov+Marynych+Samoilenko:2020, Joseph:2011}. Following \cite{Bertoin:2008} and \cite{Gnedin+Iksanov:2020}, we now provide the definition of the nested scheme. An important ingredient of the construction is a weighted branching process with positive weights which is nothing else but a multiplicative counterpart of a branching random walk.

Let $\mathbb{V}=\cup_{n\in\N_0}\N^n$ be the set of all possible
individuals of some population, encoded with the Ulam-Harris notation, where $\N_0:=\N\cup\{0\}$. The ancestor is identified with
the empty word $\varnothing$ and its weight is $P(\varnothing)=1$.
On some probability space $(\Omega, \mathcal{F}, \P)$ let
$((P_k(u))_{k\in\N})_{u\in \mathbb{V}}$ be a family of independent copies of $(P_k)_{k\in\N}$. An individual $u=u_1\ldots u_j$ of the $j$th generation whose weight is denoted by
$P(u)$ produces an infinite number of offspring residing in the $(j+1)$th generation. The
offspring of the individual $u$ are enumerated by $uk = u_1 \ldots u_j k$, where $k\in \N$, and the weights of the offspring are denoted by $P(uk)$. It is postulated that $P(uk)=P(u)P_k(u)$.
Observe that, for each $j\in\N$, $\sum_{|u|=j} P(u)=1$ a.s., where, by convention, $|u|=j$ means that the sum is taken over all individuals of the $j$th generation. For $j\in\N$, denote by
$\mathcal{F}_j$ the $\sigma$-algebra generated by the weights $(P(u))_{|u| \leq j}$.

The nested occupancy scheme in random environment is then constructed as follows. For each $j \in \N$, the set of the $j$th level boxes is identified with the set $\{u\in \mathbb{V} : |u|=j\}$ of individuals in the $j$th generation, and the weight of box $u$ is given by $P(u)$. The collection of balls is the same for all levels. The balls are allocated, conditionally on $(P(u))_{u\in\mathbb{V}}$, in the following fashion. At each time $n \in \N$, a new ball arrives and falls, independently on the $n-1$ previous balls, into box $u$ of the first level with probability $P(u)$. Simultaneously, it falls into the box $up$ of the second level with probability $P(up)/P(u)$, into the box $upq$ with probability $P(upq)/P(up)$, and so one indefinitely. A box is deemed {\it occupied} at time $n$ provided it was hit by a ball on its way over the levels. Observe that restricting attention to one arbitrary level we obtain
the occupancy scheme in the random environment.

This process can be explicitly constructed as follows. For each $j\in\N$, associated to $u \in \mathbb{\N}^j$ is the interval
\[
  I_u = \left[ \sum_{|v|=j, v \prec u} P(v) , \sum_{|v|=j, v\preccurlyeq u} P(v) \right),
\]
where $\prec$ is the alphabetical ordering\footnote{That is, $(a_1,\ldots a_j) \prec (b_1,\ldots b_j)$ if there exists $i\leq j$ such that $a_i < b_i$ and $a_r = b_r$ for all $r<i$.} on $\N^j$, and $u \preccurlyeq v$ means that either $u \prec v$ or $u=v$. Note that, for each $u \in \mathbb{V}$, the interval $I_u$ has length $P(u)$ and $(I_{uk})_{k \in\N}$ forms a partition of $I_u$. More generally, $((I_u)_{|u|=j})_{j \in \N}$ forms a nested sequence of partitions of $[0,1)$.

\begin{figure}[ht]
\centering
\begin{tikzpicture}[xscale = 11]

  \draw [thick] (0,2) -- (0,0);
  \draw [thick] (0.6,2) -- (0.6,0);
  \draw [thick] (0.8,2) -- (0.8,0);
  \draw [thick] (1,2) -- (1,0);

  \draw [<->] (0,-.3) -- (0.6,-.3);
  \draw (0.3,-.3) node[below] {$I_{1}$};

  \draw [<->] (0.6,-.3) -- (0.8,-.3);
  \draw (0.7,-.3) node[below] {$I_{2}$};

  \draw [<->] (0.8,-.3) -- (1,-.3);
  \draw (0.9,-.3) node[below] {$I_{3}$};

  \draw [thick,color=black!60] (0.4,1.33) -- (0.4,0);
  \draw [thick,color=black!60] (0.53,1.33) -- (0.53,0);
  \draw [thick,color=black!60] (0.71,1.33) -- (0.71,0);
  \draw [thick,color=black!60] (0.89,1.33) -- (0.89,0);
  \draw [thick,color=black!60] (0.96,1.33) -- (0.96,0);

  \draw [<->] (0.4,1.52) -- (0.53,1.52);
  \draw (0.465,1.52) node[above] {$I_{12}$};

  \draw [thick,color=black!40] (0.19,0.48) -- (0.19,0);
  \draw [thick,color=black!40] (0.31,0.48) -- (0.31,0);
  \draw [thick,color=black!40] (0.45,0.48) -- (0.45,0);
  \draw [thick,color=black!40] (0.48,0.48) -- (0.48,0);
  \draw [thick,color=black!40] (0.68,0.48) -- (0.68,0);
  \draw [thick,color=black!40] (0.77,0.48) -- (0.77,0);
  \draw [thick,color=black!40] (0.86,0.48) -- (0.86,0);
  \draw [thick,color=black!40] (0.99,0.48) -- (0.99,0);

  \draw [<->] (0.6,1.02) -- (0.68,1.02);
  \draw (0.64,1.02) node[above] {$I_{211}$};

  \draw[->,thick] (-0.05,0) -- (1.05,0);
  \draw [thick] (0,-0.2) node[below left] {$0$} -- (0,0.2);
  \draw [thick] (1,-0.2) node[below right] {$1$} -- (1,0.2);

  \draw [->] (0.875,3) node[above] {$U_1$} -- (0.875,1.2);
  \draw [->] (0.35,3) node[above] {$U_2$} -- (0.35,1.2);

\end{tikzpicture}
\caption{Construction of the first three levels of a nested occupancy scheme in random environment. The first ball hits boxes $3$, $31$ and $312$, the second one hits boxes $1$, $11$ and $113$.}
\end{figure}
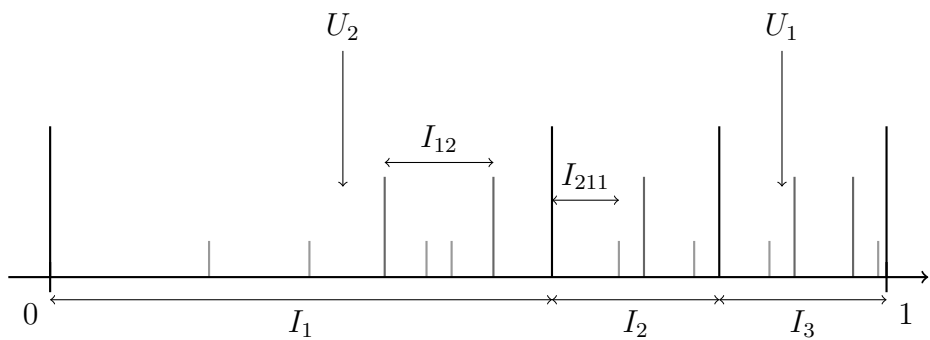

Let $(U_j)_{j\in\N}$ be independent random variables with the uniform distribution on $[0,1]$, which are independent of $(P(u))_{u \in \mathbb{V}}$. For each $n \in \N$, $U_n$ is associated to the ball $n$, which falls simultaneously in all boxes $u$ such that $U_n \in I_u$. Observe that, for all $n \in \N$, the number of balls hitting box $u$ can be written as $\sum_{i=1}^n \1_{\{U_i \in I_u\}}$. For $j, n \in \N$, put
\begin{equation}\label{as const}
  K_n^{(j)}(k):=\sum_{|u|=j}\1_{\left\{\sum_{i=1}^n \1_{\{U_i \in I_u\}} \geq k\right\}},\quad k\in\N.
\end{equation}
In words, $K_n^{(j)}(k)$ is the number of the $j$th level boxes that were hit by at least $k$ balls before time $n$. In particular, $K_n^{(j)}:=K^{(j)}_n(1)$ represents the number of occupied boxes that is, those containing at least one ball in the $j$th level.

Observe that the random variables $K_n^{(j)}(k)$ are constructed consistently for all positive integer $k$, $j$ and $n$. In particular, an immediate consequence of \eqref{as const} is that, with $n$ fixed, the sequence $(K_n^{(j)}(1))_{j \in \N}$ is a.s.\ nondecreasing, and with $j$ and $k$ fixed, the sequence $(K_n^{(j)}(k))_{n \in \N}$ is a.s.\ nondecreasing, too.

The process introduced above may be called a {\it deterministic version} of the nested occupancy scheme in random environment. Let $(N_t)_{t \geq 0}$ be a Poisson process of unit intensity with the arrival times $S_1$, $S_2,\ldots$, so that
\begin{equation*}
N_t=\#\{k\in\N: S_k\leq t\},\quad t\geq 0.
\end{equation*}
Furthermore, we assume that the Poisson process is independent of both $(P(u))_{u \in \mathbb{V}}$ and the sampling. As in much of the previous research on occupancy models, we shall also work with a {\it Poissonized version} in which balls arrive at random times $S_1$, $S_2,\ldots$ rather than $1,2,\ldots$. A key observation is that in the Poissonized nested occupancy scheme at time $t$, conditionally given $(P(u))_{u \in \mathbb{V}}$, the number of balls in box $u$ is given by a Poisson random variable of mean $tP(u)$ which is independent of the number of balls in all boxes $v$ such that $I_u \cap I_v =\emptyset$. Similarly to the deterministic version, we set
\begin{equation}\label{eq:ktdef}
\mathcal{K}_t^{(j)}(k):=K_{N_t}^{(j)}(k),\quad t\geq 0,~~k,j\in\N.
\end{equation}
The so defined random variable represents the number of boxes at time $t$ in the $j$th level of the Poissonized version containing at least $k$ balls. In particular, $\mathcal{K}_t^{(j)}:=\mathcal{K}_t^{(j)}(1)$ represents the number of occupied boxes at time $t$ in the $j$th level of the Poissonized version.

Following \cite{Iksanov+Marynych+Samoilenko:2020}, we call the $j$th level {\it early}, {\it intermediate} or {\it late} depending on whether $j$ is fixed, $j=j_n\to\infty$ and $j_n=o(\log n)$ as $n\to\infty$, or $j$ is of order $\log n$. The asymptotic behavior of $K_n^{(j)}(k)$ in certain late levels was investigated in \cite{Bertoin:2008} and \cite{Joseph:2011}. An extension of some results obtained in these two papers to a multitype version of the nested occupancy scheme in random environment can be found in~\cite{Businger:2017}. According to Theorem 1 in \cite{Joseph:2011}, the nested occupancy schemes in random environment get extinct in a level $j_\star(n)$ of order $a_\star \log n$. This means that in the level $j \geq j_\star(n)$ every ball is in a distinct box, that is, $K^{(n)}_j= n$. This shows that as far as analysis of the random variables $K_n^{(j)}(k)$ is concerned, the aforementioned classification of the levels is complete.

Our purpose is to provide a complete description of regimes of the a.s.\ convergence for $K_n^{(j)}$, properly normalized, in the late levels. Our findings complement the results obtained in \cite{Bertoin:2008}, in which the a.s.\ behavior of $K_n^{(j)}(k)$ was analyzed under the assumption $j_n = a \log n + b + o(1)$ for $a < a_\star$ and $b \in \R$. Additionally, assuming that the mean number of the first level boxes is finite we derive the asymptotic behavior of the number of empty boxes, in the regime when this number is small with respect to the total number of boxes. Last but not least, we investigate the a.s.\ convergence of $K_n^{(j)}(k)$, $k\geq 2$ in some late levels $j$. Our main technical tools are Biggins' local limit type estimate for the Gibbs measure of the branching random walk (Lemma~\ref{fct:biggStep}) and its consequences (Propositions~\ref{prop:bertoinExtended} and~\ref{prop:cltBrw}).

\subsection{Standing assumptions}\label{sect:standing}

Throughout the rest of the paper, we prefer to work with the (additive) branching random walk (BRW, for short) associated to the (multiplicative) weighted branching process
$(P(u))_{u\in\mathbb{V}}$ defined in Section~\ref{subsec:definition}. For $u \in \mathbb{V}$, put $V(u) := - \log P(u)$. By convention, $V(u) = \infty$ if $P(u) = 0$. The process $\mathcal{V} := (V(u))_{u \in \mathbb{V}}$ is a BRW, that is, in this process the daughters of each individual are positioned according to independent copies of a point process, centred around the position of that individual. The value $\infty$ is treated as a cemetery state. Individuals at that position only give birth to particles at $\infty$, and sums over $|u| = j$ ignore individuals with $V(u) = \infty$. Note that the standing assumptions $\sum_{k\geq 1}P_k=1$ a.s.\ and $\E\#\{k\in\N: P_k>0\}>1$ are equivalent to
\begin{equation}\label{eq:1}
\sum_{|u|=1} e^{-V(u)} = 1\quad \text{a.s.~~~and}\quad
\E\Big(\sum_{|u|=1} 1 \Big) > 1.
\end{equation}
We additionally impose a \emph{nonlattice} assumption, namely, for all $a>0$ and $b\in\R$,
\begin{equation}
  \label{eqn:nonLattice}
  \P(V(u) \in a \Z + b\quad\text{for all}~u~\text{with}~~|u|=1) < 1.
\end{equation}

For any $\theta\in\R$, define
\[
L(\theta): = \E\left( \sum_{|u|=1} e^{-\theta V(u)} \right)\quad \text{and} \quad
\lambda(\theta): = \log L(\theta)
\]
($\lambda$ takes values in $(-\infty,\infty]$) and then put
\[
  \underline{\theta} := \inf \{ \theta \in \R : \lambda(\theta) < \infty \}.
\]
The function $\lambda$ is strictly convex and decreasing on its domain, which particularly contains the interval $[1,\infty)$ because $\lambda(1)=0$. 

The {\it speed of BRW} is defined by $v:= - \inf_{\theta > 0} \frac{\lambda(\theta)}{\theta}$. The term is justified by the limit relation
\[
  \lim_{j \to \infty} \frac{1}{j} \min_{|u|=j} V(u) = v \quad \text{a.s.},
\]
see, for instance, Theorem 1.3 on p.~6 in \cite{Shi:2012}. We assume the existence of $\theta^*$ such that
\begin{equation}  \label{eqn:defThetaStar}
  \theta^* \lambda'(\theta^*) - \lambda(\theta^*) = 0.
\end{equation}
Under \eqref{eqn:defThetaStar}, the minimum of $\frac{\lambda(\theta)}{\theta}$ is attained at point $\theta^*$, that is,
\[
  v = -\frac{\lambda(\theta^*)}{\theta^*} = - \lambda'(\theta^*).
\]
Observe that $\lambda'(1) = -\E\sum_{|u|=1} V(u) e^{-V(u)} = \E P(u) \log P(u) \in [-\infty, 0)$, using that $P(u) \in [0,1]$ a.s. for all $|u|=1$ and $\P(P(u) = 1) < 1$. As $\lambda(1) = 0$, we infer that
\begin{equation*}
\underline{\theta} \leq 1 < \theta^* \quad \text{ and } - \lambda'(1) > v > 0.
\end{equation*}
Here, we have also used the fact that $\theta \in [1,\infty) \mapsto - \lambda^\prime(\theta)$ is strictly decreasing and positive, as $\lambda$ is strictly convex and decreasing.

Put
\[
  \lambda^*(a): = \sup_{\theta \in \R} (- \theta a - \lambda(\theta)) = - \inf_{\theta > 0}(\theta a + \lambda(\theta)),\quad a>0
\]
and note that $\lambda^*$ is the Legendre transform of $\theta \mapsto \lambda(-\theta)$. In particular, $\lambda^\ast$ is a convex function. It is a classical observation that, if there exists $\theta > \underline{\theta}$ such that $- \lambda'(\theta)=a$ for some $a>0$, then
\begin{equation}\label{eqn:formCramer}
\lambda^*(a) = -(\theta a + \lambda(\theta)) = \theta \lambda'(\theta) - \lambda(\theta).
\end{equation}
In particular, $\lambda^*(v) = 0$, with $(\lambda^*)'(v) = -\theta^*$. More generally,
\[
  \lambda'(\theta) = -a \iff (\lambda^*)'(a)=-\theta. 
\]

When $\underline{\theta}<0$, define a counterpart of $\theta^\ast$ on the negative halfline by
\begin{equation}\label{eq:thetaast-}
\theta_\ast:=\inf\{\theta\in (\underline{\theta}, 0): \theta\lambda^\prime(\theta)-\lambda(\theta)<0\}
\end{equation}
and put $\bar{a}_-:=\lim_{\theta\to \theta_\ast+}(-\lambda^\prime(\theta))$.
Two other (mutually exclusive) properties of $\lambda^\ast$ which we use later on are the following.

\noindent {\sc Property A}.
Under the assumption $\lambda(0) = \infty$, $\lambda^*$ is a strictly decreasing function. In this situation, we put
$\theta_\ast = \underline{\theta}$.

\noindent {\sc Property B}. Under the assumption $\underline{\theta} < 0$, $\lambda^*$ attains its minimum at point $-\lambda'(0)$. Furthermore, $\lambda^\ast$ is negative on $(-\lambda^\prime(0), \bar{a}_-)$.

For $\theta > \underline{\theta}$, put 
\[
  W_j(\theta)= \sum_{|u|=j} e^{-\theta V(u) - \lambda(\theta)j} \quad \text{and} \quad W(\theta) = \lim_{j \to \infty} W_j(\theta) \quad \text{a.s.}
\]
Observe that $(W_j(\theta),\mathcal{F}_j)_{j \geq 0}$ is a nonnegative martingale, so that the limit random variable $W(\theta)$ is well-defined and non-negative almost surely.

Let $\gamma>1$. Using the convexity of $x\mapsto x^\gamma$ on $\R^+:=[0,\infty)$ we obtain, for $\theta>\underline{\theta}$, $$\Big(\sum_{|u|=1}e^{-\theta V(u)}\Big)^\gamma=\Big(\sum_{|u|=1}e^{-V(u)}e^{-(\theta-1) V(u)}\Big)^\gamma\leq \sum_{|u|=1}e^{-((\theta-1)\gamma+1) V(u)}\quad\text{a.s.}$$ Choosing $\gamma>1$ sufficiently close to $1$ that satisfies $\gamma(\theta-1)+1>\underline{\theta}$, so that particularly $L(\gamma(\theta-1)+1)<\infty$, we infer $\E(W_1(\theta)^\gamma) < \infty$.
Hence, irrespective of whether $\underline{\theta} < 0$ or $\lambda(0) = \infty$, we conclude with the help of Theorem 1 in \cite{Biggins:1992} that
\[
  W(\theta) > 0\quad \text{a.s. for all $\theta \in (\theta_*,\theta^*)$ and} \quad W(\theta) = 0\quad \text{ a.s. for all $\theta > \theta^*$}.
\]
Recall that $\theta_\ast=\underline{\theta}$ if $\lambda(0) = \infty$.

\section{Main results} \label{sec:occupiedBins}

We formulate in this section our main results which are concerned with the almost sure asymptotic behaviour of $K_n^{(j)} := K_n^{(j)}(1)$ the number of occupied boxes in the late levels $j$ when $n$ balls have been thrown. More precisely, we investigate a regime in which $j=j_n$ satisfies
\[
  \log n \sim a j,\quad n\to\infty
\]
for some $a>0$. It turns out that the asymptotic behavior of $K_n^{(j)}$ depends heavily upon the value of $a$, which can be partially explained by the fact that $K_n^{(j)}\approx K_{\floor{e^{aj}}}^{(j)}$, in this setting, is an increasing function in $a$. We call the parameter $a$ the \emph{density of balls} in the $j$th level. We expect that in the levels of low density most balls fall into distinct boxes, whereas in the levels of higher density more and more balls fall into the same box, so that the total number of occupied boxes do not increase any longer proportionally to the number of balls.

Put $\Phi(x):=(2\pi)^{-1}\int_{-\infty}^x \exp(-y^2/2) {\rm d}y$ for $x\in\R$, so that $\Phi$ is the standard normal distribution function.
Put
\begin{equation}\label{eqn:defAstar}
a_* := \begin{cases}
-\frac{\lambda(2)}{2} & \text{if } \theta^*>
2,\\
v =- \frac{\lambda(\theta^*)}{\theta^*} & \text{if } \theta^* \leq 2.
\end{cases}
\end{equation}

Our first main result is Theorem~\ref{thm:occupancy}.
\begin{theorem}\label{thm:occupancy}
Let $a > 0$, $b\in\R$ and fix $\theta > 0$ such that $a = - \lambda'(\theta)$. Also, let $(j_n)_{n\in\N}$ be a sequence of positive integers.

\noindent {\rm (I)} If $a < a_\star$ and $\lim_{n\to\infty} (\log n/j_n)=a$, then $K_n^{(j_n)} = n$ {\rm a.s.} for $n$ large enough.

\noindent {\rm (II)} Let $a_\star < a < a_c := - \lambda'(1)$.

{\rm (A)} If $a< - \lambda'(2)$ and $j_n=a^{-1}\log n+ O((\log n)^{1/2})$ as $n\to\infty$, then
\begin{equation}\label{eq:impo}
K_n^{(j_n)} = n - \frac{W(2)}{2} n^2e^{\lambda(2)j_n} (1 + o(1))\quad \text{{\rm a.s.~~ as}}~~n \to \infty.
\end{equation}

{\rm (B)} If $a=-\lambda'(2)$ and
\begin{equation}\label{eq:jn}
\lim_{n\to\infty} \frac{\log n - a j_n}{(\log n)^{1/2}}=b,
\end{equation}
then as $n \to \infty$,
\begin{equation}\label{eq:impo1}
K_n^{(j_n)} = n -\frac{\Phi\big(-b (a/\lambda''(2))^{1/2}\big)}{2} W(2) n^2 e^{\lambda (2)j_n} (1+o(1))\quad \text{{\rm a.s.}}
\end{equation}

{\rm (C)} If $a > - \lambda'(2)$ and the sequence $(j_n)_{n\in\N}$ satisfies \eqref{eq:jn}, then as $n \to \infty$,
\begin{equation}\label{eq:impo2}
K_n^{(j_n)} = n - \frac{\Gamma(2-\theta)}{\theta(\theta-1)(2\pi \lambda''(\theta))^{1/2}}e^{-ab^2/(2\lambda''(\theta))} W(\theta) \frac{n^\theta e^{\lambda(\theta)j_n}}{j_n^{1/2}}   (1+ o(1))\quad {\rm a.s.}
\end{equation}

\noindent {\rm (III)} If $a = a_c$ and the sequence $(j_n)_{n\in\N}$ satisfies \eqref{eq:jn}, then
$$K_n^{(j_n)} = \Phi\big(-b (a/\lambda''(1))^{1/2}\big)n(1+o(1))\quad {\rm a.s.~~as}~~n \to \infty.$$

\noindent {\rm (IV)} If $a_c < a < \bar{a} := \lim_{\theta \to \max(\underline{\theta},0)} (-\lambda'(\theta))$ and the sequence $(j_n)_{n\in\N}$ satisfies \eqref{eq:jn}, then
\[K_n^{(j_n)} =\frac{\Gamma(1-\theta)}{\theta(2 \pi \lambda''(\theta))^{1/2}}e^{-ab^2/(2\lambda''(\theta))} W(\theta)\frac{n^\theta e^{\lambda(\theta)j_n}}{j_n^{1/2}}(1 + o(1))\quad {\rm a.s.~~as}~~n \to \infty.\]
\end{theorem}

\begin{remark}\label{rem:growth}
We intend to show that in the settings of cases (II) and (IV) of Theorem~\ref{thm:occupancy}, the a.s. growth of $n-K_n^{(j_n)}$ and $K_n^{(j_n)}$, respectively, is sublinear, the principal contribution being $n^{\alpha(a) + o(1)}$ for some $\alpha(a)\in (0,1)$. 

Let $a > 0$, we fix $\theta > 0$ such that $a = -\lambda'(\theta)$. Let $(j_n)_{n\in\N}$ be a sequence of positive integers satisfying $j_n \sim \log n/a$ as $n\to\infty$. Then
\[
  n^\theta e^{\lambda(\theta)j_n} = n^{-\frac{\lambda^*(a)}{a}+o(1)},\quad n\to\infty,
\]
which implies that $\alpha$ is given by
\[
  \alpha(a):= \begin{cases}  -\frac{\lambda^*\left(a\right)}{a} & \text{if } a> -\lambda'(2)\\ 2 - \lambda(2)/a & \text{if } a \leq - \lambda'(2) \end{cases},\quad a \in (a_\star,\bar{a}).
\]

Further, in view of \eqref{eqn:formCramer},
\[
  - \frac{\lambda^*(a)}{a} = \frac{\theta \lambda'(\theta) - \lambda(\theta)}{\lambda'(\theta)},
\]
and the function $g : \theta \in (\underline{\theta},\infty) \mapsto \theta \lambda'(\theta) - \lambda(\theta) - \lambda'(\theta)$ satisfies $g(1) = 0$ and $g'(\theta) = (\theta - 1) \lambda''(\theta)$. This in conjunction with the strict convexity of $\lambda$ implies that, for $\theta \neq 1$, $g(\theta)  > 0$ or equivalently
\[
  \theta \lambda'(\theta) - \lambda(\theta) > \lambda'(\theta).
\] Thus $\sup_{a > 0} \alpha(a) = 1$, with the supremum being attained at $a=-\lambda^\prime(1)$.
\end{remark}

\begin{remark} In case (II) of Theorem \ref{thm:occupancy}, the limit of the martingale appearing in the asymptotic expansion of $K^{(j_n)}_n$ is always positive almost surely. Let $a> a_\star$ and $\theta$ be such that $a = -\lambda'(\theta)$. By convexity of $\lambda$, there is the following dichotomy.
\begin{itemize}
  \item If $\theta^* > 2$ (in particular $W(2) > 0$ a.s.), then $a_\star = -\frac{\lambda(2)}{2} < v < -\lambda'(2)$. In this situation, the cases $\theta > 2$, $\theta=2$ and $\theta<2$ corresponds to cases (IIA), (IIB) and (IIC) of Theorem~\ref{thm:occupancy}, respectively.
  \item If $\theta^* \leq 2$, then $a_\star = v \leq -\frac{\lambda(2)}{2} \leq -\lambda'(2)$. In this situation, $\theta < \theta^*\leq 2$, so that the assumptions of case (IIC) hold.
\end{itemize}
\end{remark}

Theorem~\ref{thm:occupancy} provides the strong laws of large numbers (SLLNs, for short) for $K^{(j)}_n$ in all late levels $j=j_n$ of order $\log n$. In particular, case (I) $a < a_\star$ corresponds to the levels of the {\it very low density}, in which balls fall into distinct boxes with overwhelming probability. According to the terminology of the already mentioned article \cite{Joseph:2011}, the nested occupancy scheme gets extinct in these levels. Case (I) is treated in Section~\ref{subsec:lowDensity}. Case (II) corresponds to the levels of the {\it low density} that we call {\it presaturation levels}. The terminology stems from the fact that in these levels most balls fall into distinct boxes, but some of them may share the same box. Therefore, the SLLN for the number of occupied boxes is mainly driven by the number of available balls.  Case (II) is further divided into three subcases (IIA), (IIB) and (IIC), in which $K_n^{(j)}$ exhibits different asymptotics. Case (II) is analyzed in Section~\ref{subsec:belowSaturation}.

Case (IV) $a_c < a < \bar{a}$ corresponds to the levels of the {\it high density} that we call {\it postsaturation levels}. The motivation for this term is that in these levels most balls fall into the same collection of boxes, so that the number of occupied boxes in these levels is small with respect to the number of balls. It is explained in Remark on p.~1595 of \cite{Bertoin:2008} that the a.s.\ asymptotic behavior of $K^{(j)}_n$ is mainly driven by the number of the $j$th level boxes whose hitting probabilities are of order $e^{-j/a}\asymp 1/n$. Case (IV) which is mostly a generalization of Theorem 1 in \cite{Bertoin:2008} is treated in Section~\ref{subsec:aboveSaturation}. Case (III) $a=a_c$ corresponds to the levels of the {\it moderate density} that we call {\it saturation levels}. The a.s.\ growth rate of the number of occupied boxes in these levels is of order $cn$ for some $c \in (0,1)$. Case (III) is treated in Section~\ref{subsec:atSaturation}.

Let $k\in\N$. Next, we discuss the a.s.\ asymptotic behavior of $K^{(j_n)}_n(k)$ the number of boxes in the late levels $j_n$ which contain at least $k$ balls. Note that in low levels $j_n$ most balls cluster together. As a result, $K^{(j_n)}_n(k)$ and $K^{(j_n)}_n(1)$ are of the same order of magnitude. We focus here on the levels $j_n = \log n/(-\lambda'(k))$, in which a phase transition occurs for the asymptotic behavior of $K^{(j_n)}_n(k)$. Below that level, the behavior of $K^{(j_n)}_n(k)$ is driven by the asymptotic behavior of the number of `large' boxes (and was studied in \cite{Bertoin:2008}), whereas above that level, it is driven by the number of available balls. In some sense, Theorem~\ref{thm:occupancyk} is a generalisation to $k \in \N$ of part (III) of Theorem~\ref{thm:occupancy} dealing with the $k=1$ case.
\begin{theorem}\label{thm:occupancyk}
Let $b \in \R$, $k \in \N$ be such that $k \in (\underline{\theta},\theta^\ast)$ and $(j_n)_{n\in\N}$ a sequence of positive integers satisfying
\begin{equation*}
\lim_{n\to\infty} \frac{\log n+\lambda^\prime(k)j_n}{(\log n)^{1/2}}=b.
\end{equation*}
Then, as $n\to\infty$,
\begin{equation*}
K^{(j_n)}_n(k)=\frac{1}{k!}W(k)\Phi\left(-b\left(\tfrac{-\lambda^\prime(k)}{\lambda''(k)}\right)^{1/2}\right) n^k e^{\lambda(k) j_n} e^{-(k-1)(\lambda'(k)j_n + \log n)} (1+o(1))\quad {\rm a.s.} 
\end{equation*}
\end{theorem}

The proof of Theorem~\ref{thm:occupancyk} is given in Section~\ref{subsec:atSaturation}.

Assume that $\lambda(0) < \infty$, that is, the mean number of the first level boxes is finite. Then all boxes in low enough levels are occupied with high probability. We refer to this situation as {\it freezing}. For $j \in \N$, denote by $Z_j:=\#\{|u|=j:P(u)\neq 0\}$ the {\it number of available boxes} in the $j$th level. Then $L_n^{(j)}:=Z_j - K_n^{(j)}$ is the {\it number of empty boxes} in the $j$th level. Theorem~\ref{thm:occupancyFreezing} is a strong law of large numbers for $L_n^{(j)}$ for the levels $j$ in which freezing occurs. Recall that the quantity $\theta_\ast$ is defined in \eqref{eq:thetaast-}.
\begin{theorem}\label{thm:occupancyFreezing}
Assume that $\underline{\theta} < 0$. Let $-\lambda^\prime(0) < a < \bar{a}_- = \lim_{\theta \to \theta_\ast +} (-\lambda'(\theta))$, $b\in\R$ and pick $\theta \in (\theta_\ast, 0)$ satisfying  $a=-\lambda'(\theta)$. Also, let $(j_n)_{n\in\N}$ be a sequence of positive integers such that
\begin{equation*}
\lim_{n\to\infty} \frac{\log n - a j_n}{(\log n)^{1/2}}=b.
\end{equation*}
Then
\[
  \lim_{n \to \infty}L_n^{(j_n)}= \frac{\Gamma(-\theta)}{(2\pi\lambda^{\prime\prime}(\theta))^{1/2}}e^{-ab^2/(2\lambda''(\theta))}W(\theta)\frac{n^\theta e^{\lambda(\theta)j_n}}{j_n^{1/2}}(1+o(1))\quad {\rm a.s.~~as}~~n \to \infty.
\]
\end{theorem}

Theorem~\ref{thm:occupancyFreezing} tells us that $L_n^{(j_n)}$ diverges to $\infty$ sublinearly a.s., with the rate of divergence being smaller than $e^{\lambda(0)j_n}$, the growth rate of $Z_{j_n}$. The former claim follows by the same reasoning as given in Remark~\ref{rem:growth}. To justify the latter claim, write $n^\theta e^{\lambda(\theta)j_n}=e^{-\lambda^\ast(a)j_n(1+o(1))}$ and note that by Property B in Section~\ref{sect:standing}, $-\lambda^\ast(a)<\lambda(0)$ whenever $a\in (-\lambda^\prime(0), \bar{a}_-)$. As a consequence, the a.s.\ growth rates of $K_n^{(j_n)}$ and $Z_{j_n}$ are of the same order $e^{\lambda(0)j_n}$. This demonstrates that the theorem does indeed treat the levels $j_n$ in which freezing occurs. The proof of Theorem~\ref{thm:occupancyFreezing} is given in Section~~\ref{subsec:freezing}.

\section{Branching random walk estimates}\label{sec:asymp}

Throughout this section we work with the BRW $\mathcal{V}$ as given in Section~\ref{sect:standing}. In particular, we always assume that conditions \eqref{eq:1} and \eqref{eqn:nonLattice} hold. For all $\theta \in (\theta_*,\theta^*)$ and $j \in \N$, define the random measure
\[
  Z_\theta^{(j)} := \sum_{|u|=j} e^{ - \theta V(u) - \lambda(\theta)j} \delta_{V(u)}.
\]
Note that the total mass of $Z^{(j)}_\theta$ is $W_j(\theta)$, and that $\bar{Z}_\theta^{(j)} := Z_\theta^{(j)}/W_j(\theta)$ is a random probability measure on $\R$ which is called the \emph{Gibbs measure of the BRW with inverse temperature $\theta$}. Choosing a point according to the law $\bar{Z}_\theta^{(j)}$ can be thought of as sampling the position of an individual $u$ in the $j$th generation of the BRW with probability proportional to $e^{\theta V(u)}$. The asymptotic behavior of $\bar{Z}_\theta^{(j)}$ for $\theta \geq \theta^*$ has been explored in \cite{Madaule,Pain}. In the present work we need some uniform estimates on $Z_\theta^{(j)}$ for $\theta \in (\theta_*,\theta^*)$. These are presented in Propositions \ref{prop:bertoinExtended} and \ref{prop:cltBrw} and Corollary \ref{lem:aux22} below.

For $\theta>\theta_\ast$, put $$g_\theta(y):=\frac{1}{(2\pi \lambda^{\prime\prime}(\theta))^{1/2}}\exp\Big(-\frac{y^2}{2\lambda^{\prime\prime}(\theta)}\Big),\quad y\in\R,$$ so that $g_\theta$ is the density of centered normal distribution with variance $\lambda^{\prime\prime}(\theta)$. For $\theta \in (\theta_\ast,\theta^\ast)$, Biggins \cite{Biggins:1992} showed that the random measure $Z^{(j)}_\theta$, properly scaled, converges a.s.\ to the normal distribution with density $g_\theta$. Specifically, we recall the case $p=1$ of Theorem 4 in \cite{Biggins:1992}, in the form presented by Bertoin in Lemma~4 of~\cite{Bertoin:2008} (for a BRW satisfying \eqref{eq:1}). The result is a local limit theorem for the random measure $Z^{(j)}_\theta$ as $j\to\infty$. 
\begin{lemma}\label{fct:biggStep}
The limit relation
\begin{equation}\label{eqn:bigg1}
  \lim_{j\to\infty}\left|j^{1/2}Z^{(j)}_\theta((x-\lambda'(\theta)j-h, x-\lambda'(\theta)j+h ])-2hW(\theta)g_\theta(j^{-1/2}x)\right|=0\quad\text{{\rm a.s.}}
\end{equation}
holds uniformly in $x\in\R$, in $h$ in bounded sets, and in $\theta$ in compact subsets of $(\theta_\ast,\theta^\ast)$.
\end{lemma}

Recall that a function $f:\R \to \R^+$ is called {\it directly Riemann integrable} (dRi) on $\R$, if

\noindent (a) $\sum_{n\in\Z} \sup_{(n-1)h\leq y<nh} f(y)<\infty$ for each $h>0$ and

\noindent (b) $\lim_{h\to 0+}h \sum_{n\in\Z}\big( \sup_{(n-1)h\leq y<nh} f(y)-\inf_{(n-1)h\leq y<nh}f(y)\big)=0$.

In this article we need the following extension of Lemma~\ref{fct:biggStep}.
\begin{proposition}\label{prop:bertoinExtended}
Let $\theta\in (\theta_\ast, \theta^\ast)$ and $f : \R \to \R^+$ be a directly Riemann integrable function on $\R$. Then
\begin{multline}
  \lim_{j \to \infty} \sup_{y \in \R} \Big| \ j^{1/2} \sum_{|u|=j}e^{- \theta V(u) - \lambda(\theta) j} f(-\lambda^\prime(\theta)j+y-V(u))\\
   - W(\theta) g_\theta(yj^{-1/2})\int_\R f(x){\rm d}x \Big| = 0\quad \text{{\rm a.s.}}\label{eq:exten}
\end{multline}
\end{proposition}
\begin{remark}
Proposition~\ref{prop:bertoinExtended} is an extension of Corollary~4 in \cite{Biggins:1992} and Corollary~1 in \cite{Bertoin:2008}. The former result only treats dRi functions of compact support. The latter result is concerned with a subclass of dRi functions having a prescribed rate of decay at $\pm \infty$ and investigates the sums of the form
\[
  \sum_{|u|=j}e^{-\theta V(u)-\lambda(\theta)j}f(-\lambda^\prime(\theta)j+c_j-V(u))
\]
for sequences $(c_j)$ converging to a finite limit. Last but not least, unlike in the cited results, the convergence dealt with here is uniform in $y$. 
\end{remark}

Proposition~\ref{prop:cltBrw} given next is an a.s.\ central limit theorem for the random measure $Z_\theta^{(j)}$ as $j\to\infty$. We recall that `a.e.' is a shorthand for `almost everywhere'.
\begin{proposition}\label{prop:cltBrw}
Let $\theta\in (\theta_\ast, \theta^\ast)$ and $f : \R \to \R$ be a bounded, a.e.\ continuous function. Then, for any $C>0$, a.s.,
\[
  \lim_{j \to \infty} \sup_{|y|\leq C} \Big| \sum_{|u|=j}e^{-\theta V(u)-\lambda(\theta) j} f\Big(\frac{- \lambda^\prime(\theta)j-V(u)}{j^{1/2}}+ y\Big) - W(\theta) \int_\R f(y-x) g_\theta(x) \dd x  \Big| = 0.
\]
If, in addition, $f$ has a compact support, then the asymptotic relation holds uniformly in $y\in\R$.
\end{proposition}

A specialization of Proposition~\ref{prop:cltBrw} with $f(x)=\1_{(-\infty,\, 0]}(x)$ in combination with a simple additional argument yields the following corollary to be exploited in the proof of Lemma~\ref{lem:belowSaturationPoisson}(2).
\begin{corollary}\label{lem:aux22}
Let $\theta\in (\theta_\ast, \theta^\ast)$ and $\delta\in (0,-\lambda^\prime(\theta))$. Then, for any $C>0$, $$\lim_{j\to\infty} \sup_{|y|\leq Cj^{1/2}} \Big| \sum_{|u|=j}e^{-\theta V(u)-\lambda(\theta) j}\1_{\{V(u)\geq \delta j+y\}}-W(\theta)\Big|=0\quad\text{{\rm a.s.}}$$
\end{corollary}

The proofs of Propositions~\ref{prop:bertoinExtended} and~\ref{prop:cltBrw} and of Corollary~\ref{lem:aux22} are given in Appendix~\ref{app:proofProps}.

\section{Results for the Poissonized version. Proof of Theorem~\ref{thm:occupancy}}

All the statements of Theorem~\ref{thm:occupancy} will be proved along similar lines, with the help of a Poissonization method, that we now describe. We shall work with the random variables $\mathcal{K}_t^{(j)}(k)$, $t\geq 0$, $k,j\in\N$ defined in formula \eqref{eq:ktdef}. In other words, we shall investigate a nested occupancy scheme in random environment in which the number of balls has a Poisson distribution with mean $t$. By the thinning property of Poisson processes, conditionally on the BRW $\mathcal{V}$, the numbers of balls in different boxes of the given level are independent, and the number of balls in the box $u$ has a Poisson distribution of mean $te^{-V(u)}$.

We shall write $\rmP(\cdot):=\P(\cdot|\mathcal{V})$ for the conditional law given $\mathcal{V}$ and $\rmE$ and $\rmVar$ for the corresponding (conditional) mean and variance, respectively. It can be checked that
\begin{multline}\label{eq:mean}
\rmE \calK^{(j)}_t(k)= \sum_{|u|=j} \phi_k(te^{-V(u)}) \\ \text{ and } \rmVar \calK^{(j)}_t(k)= \sum_{|u|=j} \phi_k(te^{-V(u)}) (1 - \phi_k(te^{-V(u)})),
\end{multline}
where $\phi_k(z) = e^{-z}\sum_{i=k}^\infty \frac{z^i}{i!} = e^{-z}\left( e^z - \sum_{i=0}^{k-1} \frac{z^i}{i!} \right)$. Invoking the results of Section~\ref{sec:asymp} we shall provide the a.s. asymptotic behavior of these quantities as $t \to \infty$. Then, using the Bienaymé-Tchebychev inequality and the Borel-Cantelli lemma, we shall prove SLLNs for $\calK^{(j)}_t(k)$. At the last step called {\it depoissonization}, we shall get back to the original deterministic scheme and deduce the claimed SLLNs for $K^{(j)}_n$ from the already proved SLLNs for $\calK^{(j)}_t(k)$. The main technical tools for this final step is the SLLN for the Poisson process $N_t \sim t$ a.s. as $t \to \infty$ and the monotonicity properties of nested occupancy schemes.

\subsection{Levels of the very low density. Proof of Theorem~\ref{thm:occupancy}(I)} \label{subsec:lowDensity}

We consider in this section the number of occupied boxes in the levels of the very low density.
\begin{lemma}\label{lem:lowDensityPoisson}
Let $\epsilon > 0$ and $t_j:= e^{a_\star j - \epsilon j^{1/2}}$ for $j \in \N$. Then 
\[ \lim_{j \to \infty} \mathcal{K}^{(j)}_{t_j}(2) = 0 \quad \text{{\rm a.s.}}\]
\end{lemma}
\begin{proof}
By \eqref{eq:mean},  
\[
  \rmE \calK^{(j)}_{t_j}(2)= \sum_{|u|=j} \phi_2(e^{a_\star j-\epsilon j^{1/2}- V(u)}),\quad j\in\N,
\]
where $\phi_2(x) = 1 - e^{-x} - xe^{-x}$ for $x\geq 0$. The inequality $1-e^{-x}\leq x$, $x\geq 0$ implies that  
$\phi_2(x) \leq x^2$, $x\geq 0$.

We first assume that $\theta^* \geq 2$. Then $2a_\star = -\lambda(2)$ and, as a consequence, 
\[
  \rmE \calK^{(j)}_{t_j}(2) \leq \sum_{|u|=j} e^{2 a_\star j-2\epsilon j^{1/2}-2 V(u)}=e^{- 2\epsilon j^{1/2}}W_j(2) ,\quad j\in\N.
\]
Recall that  $(W_j(2),\mathcal{F}_j)_{j \geq 0}$ is a nonnegative martingale. Hence, $W_j(2)$ converges a.s.\ as $j \to \infty$. This entails 
\[
  \sum_{j\geq 1} \rmE \calK^{(j)}_{t_j}(2) < \infty \quad \text{a.s.}
\]
and thereupon
\[
  \lim_{j \to \infty} \calK^{(j)}_{t_j}(2) = 0 \quad \text{a.s.}
\]
by the Markov inequality and the Borel-Cantelli lemma.

We now assume that $\theta^* < 2$, in which case $a_\star = v$ and
\[
  \rmE \calK^{(j)}_{t_j}(2) \leq e^{-2\epsilon j^{1/2}}\sum_{|u|=j} e^{2(v j - V(u))},\quad j\in\N.
\]
By Lemma 3.1 in \cite{Shi:2012} applied to the BRW $(V(u)-v|u|)_{u\in\mathbb{V}}$, $\min_{|u|=j} V(u) \geq v j$ a.s.\ for all $j$ large enough. Using it in combination with $\theta^\ast<2$ and $v=-\lambda(\theta^\ast)/\theta^\ast$ we conclude that
\[
  \rmE \calK^{(j)}_{t_j}(2) \leq e^{-2\epsilon j^{1/2}} W_j(\theta^*)\quad \text{ a.s. for all $j$ large enough.}
\]
As $W_j(\theta^*)$ converges a.s.\ as $j \to \infty$, we deduce once again that $\sum_{j\geq 1} \rmE \calK^{(j)}_{t_j}(2) < \infty$ a.s. Finally, another application of the Markov inequality and the Borel-Cantelli lemma yields
\[
  \lim_{j \to \infty} \calK^{(j)}_{t_j}(2) = 0 \quad \text{a.s.}  \qedhere
\]
\end{proof}

Using Lemma~\ref{lem:lowDensityPoisson} we now prove that a.s.\ no two balls fall into the same box in the low density regime.
\begin{proof}[Proof of Theorem~\ref{thm:occupancy}(I)]
Let $(j_n)_{n\in\N}$ be a sequence of positive integers satisfying $\lim_{n \to \infty}(\log n/j_n) = a$ with $a\in (0, a_\star)$. Also, let $(t_j)_{j\in\N}$ be the sequence defined in Lemma~\ref{lem:lowDensityPoisson}. Then $\lim_{n\to\infty}(t_{j_n}/n)=\infty$, so that $N_{t_{j_n}}\geq n$ a.s.\ for $n$ large enough, by the SLLN for Poisson processes.
Since, for each $j \in \N$, the sequence $(K^{(j)}_n(2))_{n\in\N}$ is a.s. nondecreasing, we obtain a.s.\ for $n$ large enough, $$K^{(j_n)}_n(2)\leq K^{(j_n)}_{N_{t_{j_n}}}(2)=\mathcal{K}^{(j_n)}_{t_{j_n}}(2).$$ Invoking now Lemma~\ref{lem:lowDensityPoisson} we infer $\lim_{n \to \infty} K^{(j_n)}_n(2)=0$ a.s. and thereupon $K^{(j_n)}_n(2)=0$ a.s. for large enough $n$ because, for each $n\in\N$, the random variable $K^{(j_n)}_n(2)$ takes nonnegative integer values. Since, for each $n\in\N$, the sequence $(K^{(j_n)}_n(k))_{k\in\N}$ is a.s. nonincreasing, we infer that, for all $k\geq 2$, $K^{(j_n)}_n(k)=0$ a.s. for $n$ large enough. This in combination with the equality $\sum_{k\geq 1} K^{(j_n)}_n(k)=n$ a.s. which holds for each $n$ leads to the desired conclusion $K^{(j_n)}_n(1) =n$ a.s. for $n$ large enough.

\end{proof}

\subsection{Presaturation levels. Proof of Theorem~\ref{thm:occupancy}(II)}
\label{subsec:belowSaturation}

We consider in this section the number of occupied boxes in the presaturation levels. Since in these levels most balls fall into distinct boxes, it is natural to investigate the asymptotic behavior of the difference between the number of balls and the number of occupied boxes.

\begin{lemma} \label{lem:belowSaturationPoisson}
Let $a\in (a_\star, -\lambda'(1))$ and $t_j(y) = e^{a j + y}$ for $j\in\N$ and $y\in\R$. Pick $\theta > 1$ satisfying $a = -\lambda'(\theta)$.

\noindent {\rm (1)} If $a< -\lambda^\prime(2)$ (equivalently $\theta>2$),
then, for any $C>0$,
\[
  \lim_{j \to \infty} \sup_{|y|\leq Cj^{1/2}} \Big| (t_{j}(y))^{-2} e^{-\lambda(2)j}\rmE\left( N_{t_j(y)} - \calK^{(j)}_{t_j(y)}\right) - \frac{1}{2} W(2)  \Big| = 0 \quad \text{{\rm a.s.}},
\]
\[
  \lim_{j \to \infty} \sup_{|y|\leq Cj^{1/2}} \Big|(t_{j}(y))^{-2} e^{-\lambda(2)j}\rmVar\left( N_{t_j(y)} - \calK^{(j)}_{t_j(y)}\right) - \frac{1}{2} W(2)  \Big| = 0 \quad \text{{\rm a.s.}}
\]
$$\text{and}\quad N_{t_j(y)} - \calK^{(j)}_{t_j(y)}=\frac{W(2)}{2}(t_{j}(y))^2 e^{\lambda (2)j}(1+o(1))\quad \text{{\rm a.s.}~~ as}~~ j \to \infty,$$ uniformly in $y$ satisfying $|y|\leq Cj^{1/2}$.

\noindent {\rm (2)} If $a=-\lambda^\prime(2)$ (equivalently $\theta=2$), then, for any $C>0$,
\begin{multline*}
\lim_{j \to \infty} \sup_{|y|\leq Cj^{1/2}} \Big|(t_{j}(y))^{-2} e^{-\lambda(2)j} \rmE\left( N_{t_j(y)} - \calK^{(j)}_{t_j(y)}\right)\\- \frac{1}{2}  W(2) \Phi\big(-y(\lambda''(2)j)^{-1/2}\big)
\Big| = 0 \quad \text{{\rm a.s.}},
\end{multline*}
\begin{multline*}
  \lim_{j \to \infty} \sup_{|y|\leq Cj^{1/2}} \Big|(t_{j}(y))^{-2} e^{-\lambda(2)j}\rmVar\left( N_{t_j(y)} - \calK^{(j)}_{t_j(y)}\right)\\-\frac{1}{2} W(2) \Phi\big(-y(\lambda''(2)j)^{-1/2}\big)   \Big| = 0 \quad \text{{\rm a.s.}}
\end{multline*}
$$\text{and} \quad N_{t_j(y)} - \calK^{(j)}_{t_j(y)}=\frac{W(2)}{2}\Phi\big(-y(\lambda''(2)j)^{-1/2}\big)(t_{j}(y))^2 e^{\lambda (2)j}(1+o(1))\quad \text{\rm a.s.}$$
as $j\to\infty$, uniformly in $y$ satisfying $|y|\leq Cj^{1/2}$.

\noindent {\rm (3)} If $a>-\lambda^\prime(2)$ (equivalently $\theta<2$), then
\begin{multline*}
  \lim_{j \to \infty} \sup_{y \in \R} \Big| j^{1/2} (t_{j}(y))^{-\theta} e^{-\lambda(\theta)j}\rmE\left( N_{t_j(y)} - \calK^{(j)}_{t_j(y)}\right)\\-\frac{\Gamma(2-\theta)}{\theta(\theta-1)} W(\theta)g_\theta(yj^{-1/2})\Big| = 0 \quad \text{{\rm a.s.}},
\end{multline*}
\begin{multline*}
\lim_{j \to \infty} \sup_{y \in \R} \Big| j^{1/2} (t_{j}(y))^{-\theta} e^{-\lambda(\theta)j}\rmVar\left( N_{t_j(y)} - \calK^{(j)}_{t_j(y)}\right) \\
-(2\theta-2^\theta+1) \frac{\Gamma(2-\theta)}{\theta(\theta-1)} W(\theta)g_\theta(yj^{-1/2})\Big| = 0 \quad \text{{\rm a.s.},}
\end{multline*}
\[
  \text{and} \quad N_{t_j(y)} - \calK^{(j)}_{t_j(y)}=\frac{\Gamma(2-\theta)}{\theta(\theta-1)} W(\theta)g_\theta(yj^{-1/2})j^{-1/2} (t_{j}(y))^\theta e^{\lambda (\theta)j}(1+o(1))\quad \text{\rm a.s.}
\]
as $j \to \infty$, uniformly in $y$ satisfying $|y|\leq Cj^{1/2}$, for any $C>0$. 
\end{lemma}
\begin{proof}
Let $\eta_\lambda$ be a random variable with the Poisson distribution of mean $\lambda>0$. Then
\[
  \E(\eta_\lambda- \1_{\{\eta_\lambda\geq 1\}}) = m(\lambda) \quad \text{and} \quad \Var(\eta_\lambda-\1_{\{\eta_\lambda\geq 1\}}) = v(\lambda),
\]
where $m(x):=x -1+e^{-x}$ and $v(x):=x + e^{-x} - e^{-2x} - 2 xe^{-x}$ for $x\geq 0$. In view of these formulas, for $j \in \N$ and $y \in \R$,
\begin{equation}\label{eqn:diffBinsBalls}
\rmE\big( N_{t_j(y)} - \calK^{(j)}_{t_j(y)}\big) = \sum_{|u|=j} m\big(e^{aj+y-V(u)}\big)
\end{equation}
and
\begin{equation}\label{eqn:diffBinsBallsVar}
\rmVar\big( N_{t_j(y)} - \calK^{(j)}_{t_j(y)}\big)=\sum_{|u|=j} v\big(e^{aj+y-V(u)}\big).
\end{equation}

\noindent {\sc Case $\theta\geq 2$}. For typographical ease, we write $\vartheta$ for the integer part of $\theta$. We start with
\begin{align}\nonumber
\rmE\big( N_{t_j(y)} - \calK^{(j)}_{t_j(y)}\big) &= \sum_{i=2}^{\vartheta-1} \sum_{|u|=j}\frac{(-1)^i e^{i(aj+y-V(u))}}{i!}+ \sum_{|u|=j} \bar{m}_\vartheta(e^{aj+y-V(u)})\\
&=\sum_{i=2}^{\vartheta-1} \frac{(-1)^i}{i!} e^{i(aj+y)+ \lambda(i)j}W(i) + \sum_{|u|=j} \bar{m}_\vartheta(e^{aj+y-V(u)}), \label{eqn:decomp}
\end{align}
where $\bar{m}_\vartheta(x):=e^{-x} - \sum_{i=0}^{\vartheta-1} \frac{(-x)^i}{i!}$ for $x\geq 0$. By convention, the empty sum $\sum_{i=2}^1 \ldots$ is equal to $0$, if $\theta \in (2,3)$. Further,
\begin{multline*}
\sum_{|u|=j} \bar{m}_\vartheta(e^{aj+y-V(u)})
=\frac{(-1)^\vartheta}{\vartheta!}e^{\vartheta(aj+y)+\lambda(\vartheta)j} \sum_{|u|=j}e^{-\vartheta V(u)-\lambda(\vartheta)j}\1_{\{V(u) \geq aj+y\}}\\ + e^{\theta(aj+y) + \lambda(\theta)j}
\sum_{|u|=j} e^{-\theta V(u)-\lambda(\theta)j}\hat{m}_\vartheta(e^{aj + y - V(u)}),
\end{multline*}
where $\hat{m}_\vartheta(x):=x^{-\theta}\big(e^{-x} - \sum_{i=0}^{\vartheta-1} \frac{(-x)^i}{i!} - (-1)^\vartheta \frac{x^\vartheta}{\vartheta!}\1_{[0, 1]}(x)\big)$ for $x>0$.
The function $x\mapsto \hat{m}_\vartheta(e^x)$ is Lebesgue integrable on $\R$. For instance,
\[0\leq \int_{-\infty}^0 (-1)^{\vartheta+1}\hat{m}_\vartheta(e^x){\rm d}x=\int_0^1(-1)^{\vartheta+1} x^{-\theta-1}\Big(e^{-x}-\sum_{i=0}^\vartheta \frac{(-x)^i}{i!}\Big) {\rm d}x<\infty\]
because
\[
  (-1)^{\vartheta+1}x^{-\theta-1}\Big(e^{-x}-\sum_{i=0}^\vartheta \frac{(-x)^i}{i!}\Big)~\sim~ \frac{1}{(\vartheta+1)!x^{\{\theta\}}},\quad x\to 0+,
\]
where $\{\theta\}$ is the fractional part of $\theta$. Further, the function $x\mapsto (-1)^\vartheta e^{\theta x} \hat{m}_\vartheta(e^x)$ is nonnegative and increasing on $[0,\infty)$, and the function $x\mapsto (-1)^{\vartheta+1} e^{\theta x}\hat{m}_\vartheta(e^x)$ is nonnegative and increasing on $(-\infty, 0]$. Hence, we conclude with the help of Lemma 9.1 in \cite{Goldie:1991} that the functions $x\mapsto (-1)^\vartheta \hat{m}_\vartheta(e^x)$ and $x\mapsto (-1)^{\vartheta+1}\hat{m}_\vartheta(e^x)$ are dRi on $[0,\infty)$ and $(-\infty, 0]$, respectively. An application of Proposition~\ref{prop:bertoinExtended} with $f(x)=(-1)^\vartheta \hat{m}_\vartheta(e^x)\1_{[0,\infty)}(x)$ and then with $f(x)=(-1)^{\vartheta+1}\hat{m}_\vartheta(e^x)\1_{(-\infty, 0)}(x)$ yields
\begin{multline}\label{eq:aux3}
\lim_{j \to \infty}  \sup_{y \in \R} \Big|  j^{1/2} \sum_{|u|=j} e^{-\theta V(u)-\lambda(\theta)j}
\hat{m}_\vartheta(e^{aj+y-V(u)})\\ -W(\theta) g_\theta(y)\int_0^\infty \hat{m}_\vartheta(e^{-x}) \dd x\Big| = 0 \quad \text{a.s.}
\end{multline}
Assume that $\theta$ is integer, so that $\vartheta=\theta$. Then, by Proposition~\ref{prop:cltBrw} applied to $f(x)=\1_{(-\infty, 0]}(x)$, $x\in\R$, for any $C>0$,
\begin{equation}\label{eq:aux23}
\lim_{j \to \infty}  \sup_{|y|\leq Cj^{1/2}} \Big|  \sum_{|u|=j}e^{-\vartheta V(u)-\lambda(\vartheta)j}
\1_{\{V(u)\geq aj+y\}}-W(\vartheta) \Phi\big(-y (\lambda''(\vartheta)j)^{-1/2}\big) \Big| = 0 \quad \text{a.s.}
\end{equation}
Assume now that $\theta$ is not integer, so that $\vartheta<\theta$ and thereupon $-\lambda^\prime(\vartheta)>-\lambda^\prime(\theta)=a>0$. Here, the left-hand inequality is justified by the fact that $-\lambda^\prime$ is a decreasing function on $(\theta_\ast,\infty)$. Invoking Corollary~\ref{lem:aux22} with $\delta=a=-\lambda^\prime(\theta)\in (0,-\lambda^\prime(\vartheta))$ we obtain, for any $C>0$,
\begin{equation}\label{eq:aux24}
\lim_{j \to \infty} \sup_{|y|\leq Cj^{1/2}} \Big|  \sum_{|u|=j}e^{-\vartheta V(u)-\lambda(\vartheta)j}
\1_{\{V(u)\geq aj+y\}}-W(\vartheta)\Big| = 0 \quad \text{a.s.}
\end{equation}

Combining \eqref{eq:aux3} and \eqref{eq:aux23} we obtain in the case of integer $\theta$ ($\vartheta=\theta$)
\begin{multline}
\lim_{j \to \infty} \sup_{|y|\leq Cj^{1/2}} \Big|e^{-\vartheta(aj+y)-\lambda(\vartheta)j}\sum_{|u|=j} \bar{m}_\vartheta(e^{aj+y-V(u)})\\-\frac{(-1)^\vartheta}{\vartheta!} W(\vartheta)\Phi\big(-y(\lambda''(\vartheta)j)^{-1/2}\big)\Big| = 0 \quad \text{a.s.}\label{eq:later}
\end{multline}
If $\theta$ is not integer, we conclude with the help of \eqref{eq:aux3} and \eqref{eq:aux24} that
\begin{equation}
\lim_{j \to \infty} \sup_{|y|\leq Cj^{1/2}} \Big|e^{-\vartheta(aj+y)-\lambda(\vartheta)j}\sum_{|u|=j} \bar{m}_\vartheta(e^{aj+y-V(u)})-\frac{(-1)^\vartheta}{\vartheta!} W(\vartheta)\Big| = 0 \quad \text{a.s.}\label{eq:later21}
\end{equation}
having utilized the fact that the function $s\mapsto as+\lambda(s)$ is decreasing on $[2,\theta]$.

Getting back to \eqref{eqn:decomp} we first analyze the case $\theta=2$ in which the first summand on the right-hand side of \eqref{eqn:decomp} vanishes. Appealing to \eqref{eq:later} with $\vartheta=2$ we obtain
\begin{multline*}
\lim_{j \to \infty} \sup_{|y|\leq Cj^{1/2}}\Big|e^{-2(aj+y)- \lambda(2)j}\rmE\left( N_{t_j(y)} - \calK^{(j)}_{t_j(y)}\right)\\
-W(2)\Phi\big(-y (\lambda''(2)j)^{-1/2}\big)/2\Big| = 0 \quad \text{a.s.}
\end{multline*}
Assume now that $\theta\in (2,3)$ in which case the first summand on the right-hand side of \eqref{eqn:decomp} still vanishes. Invoking \eqref{eq:later21} with $\vartheta=2$ we arrive at
\begin{equation}
\lim_{j \to \infty} \sup_{|y|\leq Cj^{1/2}}\Big|e^{-2(aj+y)- \lambda(2)j}\rmE\left( N_{t_j(y)} - \calK^{(j)}_{t_j(y)}\right)-W(2)/2\Big| = 0 \quad \text{a.s.}\label{eq:later22}
\end{equation}
Finally, if $\theta\geq 3$, then the first summand on the right-hand side of \eqref{eqn:decomp} is nonzero. Furthermore, in view of the already mentioned monotonicity of $s\mapsto as+\lambda(s)$ on $[2,\theta]$, its first term which corresponds to $i=2$ dominates. As a consequence, relation \eqref{eq:later22} holds but the driving forces behind this convergence are different from those in the case $\theta\in (2,3)$.

Passing to the analysis of the variance we write $v(x)=m(x)+w(x)$, where $w(x):=1-e^{-2x}-2xe^{-x}$ for $x\geq 0$. Using now \eqref{eqn:diffBinsBalls} and \eqref{eqn:diffBinsBallsVar} we conclude that
$$\rmVar\big( N_{t_j(y)} - \calK^{(j)}_{t_j(y)}\big)=\rmE\big( N_{t_j(y)} - \calK^{(j)}_{t_j(y)}\big)+\sum_{|u|=j}w\big(e^{aj+y-V(u)}\big).$$ 
Hence, in order to show that the variance exhibits (uniformly) the same asymptotics as the mean it is enough to check that, for any $C>0$,
\begin{equation}\label{eq:aux4}
\lim_{j\to\infty} \sup_{|y|\leq Cj^{1/2}} e^{-2(aj+y)- \lambda(2)j}\sum_{|u|=j} w\big(e^{aj+y-V(u)}\big)= 0\quad \text{a.s.}
\end{equation}

Since, as $x\to 0+$, $m(x)\sim x^2/2$ and $w(x)\sim x^3/3$, we infer $$\lim_{x\to 0+}(w(x)/m(x))=0.$$ Thus, given $\varepsilon>0$ there exists $x_0\in\R$ such that $w(e^x)\leq \varepsilon m(e^x)$ whenever $x\leq x_0$. Using this $x_0$ we write
\begin{multline*}
\sum_{|u|=j} w\big(e^{aj+y-V(u)}\big)=\sum_{|u|=j} w\big(e^{aj+y-V(u)}\big)\1_{\{aj+y-V(u)\leq x_0\}}\\+\sum_{|u|=j} w\big(e^{aj+y-V(u)}\big)\1_{\{aj+y-V(u)>x_0\}}.
\end{multline*}
We shall investigate the summands separately:
\begin{multline}\label{eq:aux12}
\sum_{|u|=j} w\big(e^{aj+y-V(u)}\big)\1_{\{aj+y-V(u)\leq x_0\}}\leq \varepsilon \sum_{|u|=j} m\big(e^{aj+y-V(u)}\big)\\~\sim~ \varepsilon W(2)c_je^{2(aj+y)+\lambda(2)j},\quad j\to\infty,
\end{multline}
uniformly in $|y|\leq Cj^{1/2}$, for any $C>0$. Here, $c_j=1/2$ if $\theta=2$ and $c_j=\Phi(-y(\lambda''(2)j)^{-1/2})/2$ if $\theta=2$, and the last limit relation follows from the already proved part for the mean. To analyze the second summand, put $\hat{w}(x):= x^{-\theta}w(x)\1_{(e^{x_0}, \infty)}(x)$ for $x>0$. The function $x \mapsto \hat{w}(e^x)$ is nonnegative. Further, since $\hat{w}(e^x)\sim e^{-\theta x}$ as $x\to\infty$, it is Lebesgue integrable on $\R$. As $x \mapsto e^{\theta x} \hat{w} (e^x)$ is increasing on $\R$, using again Lemma 9.1 in \cite{Goldie:1991}, we conclude that $x \mapsto \hat{w}(e^x)$ is dRi on $\R$. Noting that $$\sum_{|u|=j} w\big(e^{aj+y-V(u)}\big)\1_{\{aj+y-V(u)>x_0\}}=e^{\theta(aj+y)+\lambda(\theta)j}\sum_{|u|=j}\hat{w}\big(e^{aj+y-V(u)}\big),$$ an application of Proposition~\ref{prop:bertoinExtended} with $f(x)=\hat{w}(e^x)$ yields
\begin{multline*}
\lim_{j \to \infty} \sup_{y \in \R} \Big| j^{1/2} e^{-\theta(aj+y)- \lambda(\theta)j}\sum_{|u|=j} w\big(e^{aj+y-V(u)}\big)\1_{\{aj+y-V(u)>x_0\}}\\ - W(\theta) g_\theta(yj^{1/2}) \int_\R \hat{w}(e^x) \dd x \Big|=0\quad \text{a.s.}
\end{multline*}
Recalling that the function $s\mapsto as+\lambda(s)$ is decreasing on $[2,\theta]$ we conclude that
\[
  \lim_{j \to \infty} \sup_{y \in \R} e^{-2(aj+y)- \lambda(2)j}\sum_{|u|=j} w\big(e^{aj+y-V(u)}\big)\1_{\{aj+y-V(u)>x_0\}}=0\quad \text{a.s.}
\]
This in combination with \eqref{eq:aux12} proves \eqref{eq:aux4}.

\noindent {\sc Case $\theta\in (1,2)$}.
We use a representation
\[
  \sum_{|u|=j} m\left( e^{aj+j-V(u)}\right) = e^{\theta (aj+y)+ \lambda(\theta)j} \sum_{|u|=j} e^{-\theta V(u)-\lambda(\theta)j} \frac{m\left( e^{aj+y-V(u)}\right)}{e^{\theta(aj+y-V(u))}}.
\]
Observe that the function $x\mapsto e^{-\theta x}m(e^x)$ is nonnegative and that $$\int_{\R}e^{-\theta x}m(e^x){\rm d}x=\int_0^\infty x^{-\theta-1}(x-1+e^{-x}){\rm d}x=\frac{\Gamma(2-\theta)}{\theta(\theta-1)}$$ which can be checked by repeated integration by parts. Since $x\mapsto m(e^x)$ is an increasing function on $\R$, we conclude that the function $x\mapsto e^{-\theta x}m(e^x)$ is dRi on $\R$ by Lemma 9.1 in \cite{Goldie:1991}. Invoking Proposition~\ref{prop:bertoinExtended} with $f(x)=e^{-\theta x}m(e^x)$ we obtain, a.s.
\[
  \lim_{j \to \infty} \sup_{y \in \R} \left| j^{1/2}\sum_{|u|=j} e^{-\theta V(u) - \lambda(\theta)j} \frac{m\left( e^{aj+y-V(u)}\right)}{e^{\theta(aj+y-V(u))}}-W(\theta)g_\theta(yj^{-1/2}) \int_\R \frac{m(e^x)}{e^{\theta x}} \dd x \right|=0 
\]
and thereupon, a.s.
\[
  \lim_{j \to \infty} \sup_{y \in \R} \left| j^{1/2} e^{-\theta(aj+y)- \lambda(\theta)j}\rmE\left( N_{t_j(y)} - \calK^{(j)}_{t_j(y)}\right)-\frac{\Gamma(2-\theta)}{\theta(\theta-1)}W(\theta) g_\theta(yj^{-1/2}) \right|=0.
\]
Recalling the representation $v(x)=m(x)+w(x)$ for $x\geq 0$ and noting that the function $x\mapsto e^{-\theta x}w(x)$ is dRi on $\R$ with
\begin{multline*}
\int_\R e^{-\theta x}w(e^x){\rm d}x=\int_0^\infty x^{-\theta-1}(1-e^{-2x}-2xe^{-x}){\rm d}x=2\int_0^\infty x^{-\theta}(1-e^{-x}){\rm d}x\\-\int_0^\infty x^{-\theta-1}(2x-1+e^{-2x}){\rm d}x=(2\theta -2^\theta)\frac{\Gamma(2-\theta)}{\theta(\theta-1)},
\end{multline*}
another appeal to Proposition~\ref{prop:bertoinExtended} yields
\begin{multline*}
\lim_{j \to \infty} \sup_{y \in \R} \Big| j^{1/2} e^{-\theta(aj+y)-\lambda(\theta)j}\rmVar\left( N_{t_j(y)} - \calK^{(j)}_{t_j(y)}(1) \right)\\ -W(\theta)
g_\theta(yj^{-1/2}) \int_\R e^{-\theta x}(m(e^x)+w(e^x))\dd x \Big|=0 \quad \text{a.s.}
\end{multline*}

It remains to prove the SLLNs. To this end, we treat the three cases simultaneously and write, for any $\rho > 0$ and $y$ satisfying $|y|\leq Cj^{1/2}$, for any fixed $C>0$ when $\theta\geq 2$ and $y=y_j\in\R$ when $\theta\in (1,2)$,
\begin{multline}
\rmP\left\{ \left| N_{t_j(y)} - \calK^{(j)}_{t_j(y)}-\rmE\left( N_{t_j(y)} - \calK^{(j)}_{t_j(y)}\right) \right|>\rho \rmE\left( N_{t_j(y)} - \calK^{(j)}_{t_j(y)}\right)\right\}\\ \leq \frac{{\rm Var}\,(N_{t_j(y)} - \calK^{(j)}_{t_j(y)})}{\rho^2 (\E(N_{t_j(y)} - \calK^{(j)}_{t_j(y)}))^2}~\sim~\frac{C_\theta}{\rho^2(\E(N_{t_j(y)} - \calK^{(j)}_{t_j(y)}))},\quad j\to\infty.\label{eq:aux5}
\end{multline}
We have used the Bienaymé-Chebyshev inequality and then the already proved asymptotics of the (conditional) mean and variance. Here, $C_\theta=1$ when $\theta\geq 2$ and $C_\theta=2\theta-2^\theta+1$ when $\theta\in (1,2)$. Invoking once again the already proved part of the lemma we conclude that in either of the three cases the right-hand side in \eqref{eq:aux5} is the general term of a convergent series in $j$. Hence, by the Borel-Cantelli lemma, the strong laws of large numbers hold conditionally on the BRW $\mathcal{V}$, hence also unconditionally.
\end{proof}

The SLLNs for the Poissonized version given in Lemma~\ref{lem:belowSaturationPoisson} will now be used to prove the SLLNs for the number of occupied boxes in the presaturation levels of the nested occupancy scheme with $n$ balls.
\begin{proof}[Proof of Theorem~\ref{thm:occupancy}(II)]
Let $a\in (a_\star, -\lambda'(1))$, $b\in\R$. Fix $\delta > 0$ and note that, for each $j \in \N$, the sequence $(n - K^{(j)}_n)_{n\in\N}$ is a.s.\ nondecreasing. Moreover, by the SLLN for Poisson processes, $N_{(1-\delta)n}\leq n\leq N_{(1+\delta)n}$ a.s.\ for $n$ large enough. As a consequence, a.s.\ for $n$ large enough,
\begin{equation}\label{eq:4}
N_{(1-\delta)n} - \calK^{(j_n)}_{(1-\delta)n} \leq n - K^{(j_n)}_n \leq N_{(1+\delta)n} - \calK^{(j_n)}_{(1+\delta)n}.
\end{equation}
From now on, we investigate the three cases separately. Fix any $\gamma>0$.

\noindent {\sc Case $a<-\lambda^\prime(2)$}. We shall apply the SLLN from part (1) of Lemma~\ref{lem:belowSaturationPoisson} with $j=j_n=a^{-1}\log n+O((\log n)^{1/2})$ and $y=\log \gamma-(j_n - a^{-1}\log n)$.
The so defined $y$ satisfies $|y|\leq Cj_n^{1/2}$ for some $C>0$ and large $n$ and $t_{j_n}(y)=\gamma n$. Thus,
$$\lim_{n\to\infty} \frac{N_{\gamma n} - \calK^{(j_n)}_{\gamma n}}{n^2 e^{\lambda(2)j_n}}=\frac{\gamma^2}{2}W(2)\quad \text{a.s.}$$
Using this in combination with \eqref{eq:4} we arrive at \eqref{eq:impo}.

\noindent {\sc Case $a=-\lambda^\prime(2)$}. We shall apply the SLLN from part (2) of Lemma~\ref{lem:belowSaturationPoisson} with
\begin{align*}
  &j=j_n=a^{-1}\log n+a^{-1}b(\log n)^{1/2}+o((\log n)^{1/2})\\
   \text{ and }\quad &y=\log \gamma+b(\log n)^{1/2}-(j_n -a^{-1}\log n-a^{-1}b(\log n)^{1/2} ).
\end{align*}
The so defined $y=y_n$ satisfies $|y|\leq Cj_n^{1/2}$ for some $C>b$ and large $n$ and $t_{j_n}(y)=\gamma n$. Furthermore, $\lim_{n\to\infty} (y_n j_n^{-1/2})=a^{1/2}b$. Hence, $$\lim_{n\to\infty} \frac{N_{\gamma n} - \calK^{(j_n)}_{\gamma n}}{n^2 e^{\lambda(2)j_n}}=\frac{\gamma^2}{2}W(2)\Phi(-b (a/\lambda''(2))^{1/2})\quad \text{a.s.}$$ This together with \eqref{eq:4} proves \eqref{eq:impo1}.

\noindent {\sc Case $a> -\lambda^\prime(2)$}. Using \eqref{eq:4} in combination with the SLLN from part (3) of Lemma~\ref{lem:belowSaturationPoisson} with the same $j$ and $y$ as in the previous case yields \eqref{eq:impo2}.
\end{proof}

\subsection{Saturation levels. Proof of Theorem~\ref{thm:occupancy}(III)}\label{subsec:atSaturation}

We investigate in this section the a.s.\ asymptotic behavior of the number of occupied boxes in the saturation levels. More precisely, the number $n$ of balls thrown and the levels $j=j_n$ satisfy $\log n \sim - \lambda'(1) j_n$ as $n\to\infty$. In this case, a positive fraction of balls share the same box, and the variables $K_n^{(j_n)}$ are no longer proportional to $n$. Levels around the saturation levels of order $\log n/(-\lambda'(1))$ form a transition for the nested occupancy scheme, between the sparse phases (I) and (II) and the dense phase (IV).

Let $k\in\N$, $k\geq 2$ and assume that $\lambda^\prime(k)$ is finite. More generally, a similar transition occurs in the levels of order $\log n/(-\lambda'(k))$ for the number of boxes containing at least $k$ balls. The precise statement is given in Theorem~\ref{thm:occupancyk}.

For $\theta\in (\underline{\theta},\theta^\ast)$, put $\varphi(\theta):=\lambda(\theta)-\theta \lambda^\prime(\theta)$. Recall that the covariance of square integrable random variables $X$ and $Y$ is given by ${\rm Cov}\,(X,Y)=\E XY-\E X \E Y$. We write $\rmCov$ for the conditional covariance given the BRW $\mathcal{V}$. Formula \eqref{eq:covar} which describes the a.s. asymptotic behavior of the conditional covariance is not used in the present paper. It is given for completeness.
\begin{lemma}\label{lem:belowSaturationPoissonk}
Let $k\in \N$ such that $k\in (\underline{\theta}, \theta^\ast)$ and put $t^{(k)}_j(y)=-\lambda^\prime(k)j+y$. Then, for any $C>0$,
\[
  \lim_{j \to \infty} \sup_{|y|\leq Cj^{1/2}} \Big|e^{-(\varphi(k)j+y)} \rmE \calK^{(j)}_{t^{(k)}_j(y)}(k)
- (k!)^{-1} \Phi\big(-y(j \lambda''(k))^{-1/2}\big)W(k) \Big| = 0 \quad \text{{\rm a.s.}},
\]
\[
  \lim_{j \to \infty} \sup_{|y|\leq Cj^{1/2}} \Big| e^{-(\varphi(k)j+y)}\rmVar \calK^{(j)}_{t^{(k)}_j(y)}(k)- (k!)^{-1} \Phi\big(-y(j \lambda''(k))^{-1/2}\big)W(k) \Big| = 0 \quad \text{{\rm a.s.}}
\]
and
\begin{equation}\label{eq:slln}
\calK^{(j)}_{t^{(k)}_j(y)}(k)=(k!)^{-1}W(k)\Phi\big(-y(j \lambda''(k))^{-1/2}\big)e^{\varphi(k)j+y}(1+o(1))\quad \text{{\rm a.s.}~~ as}~~ j \to \infty
\end{equation}
uniformly in $y$ satisfying $|y|\leq Cj^{1/2}$.

Let positive integers $\ell<k$ be such that $\ell, k\in (\theta_\ast, \theta^\ast)$. Then, for any $C>0$,
\begin{multline}\label{eq:covar}
 \lim_{j \to \infty} \sup_{|y|\leq Cj^{1/2}} \Big|e^{-(\varphi(k)j+y)}\rmCov\, (\calK^{(j)}_{t^{(k)}_j(y)}(\ell), \calK^{(j)}_{t^{(k)}_j(y)}(k))\\
- (k!)^{-1} \Phi\big(-y(j \lambda''(k))^{-1/2}\big)W(k) \Big| = 0 \quad \text{{\rm a.s.}}
\end{multline}
\end{lemma}

\begin{proof}
Recall that, for $k\in\N$, $$\phi_k(x) = 1 - e^{-x} \sum_{i=0}^{k-1} \frac{x^i}{i!},\quad x\geq 0$$ and put
$h_k(x):=e^{-kx} \phi_k(e^x)$ for $x\in\R$.

While dealing with the mean we always assume that integer $k\in (\underline{\theta},\theta^\ast)$. Invoking \eqref{eq:mean} we obtain
\begin{multline*}
\rmE \calK^{(j)}_{t^{(k)}_j(y)}(k)=\sum_{|u|=j}\phi_k(e^{-\lambda^\prime(k)+y-V(u)})=e^{\varphi(k)j+y}\\ \times\Big(\sum_{|u|=j}e^{-kV(u)-\lambda(k)j}h_k(-\lambda^\prime(k)+y-V(u))\1_{\{V(u)\leq -\lambda^\prime(k)+y \}}\\- \sum_{|u|=j}e^{-kV(u)-\lambda(k)j}\Big(\frac{1}{k!}-h_k(-\lambda^\prime(k)+y-V(u))\Big)\1_{\{V(u)>-\lambda^\prime(k)+y\}}\\+\frac{1}{k!}\sum_{|u|=j} e^{-kV(u)-\lambda(k)j}\1_{\{V(u)> -\lambda^\prime(k)+y\}}\Big).
\end{multline*}
The function $h_k$ is nonincreasing on $\R$ because $h_k(\log x)=(k-1)!\int_0^1 e^{-xy}y^{k-1}{\rm d}y$ for $x>0$. Since $h_k(x)\sim e^{-kx}$ as $x\to\infty$, it is Lebesgue integrable on $\R^+$, hence dRi on $\R^+$. An application of Proposition~\ref{prop:bertoinExtended} with $f(x)=h_k(x)\1_{[0,\infty)}(x)$ yields
\begin{multline*}
\lim_{j \to \infty}  \sup_{y \in \R} \Big|  j^{1/2} \sum_{|u|=j}e^{-kV(u)-\lambda(k)j}h_k(-\lambda^\prime(k)j+y-V(u))\1_{\{V(u)\leq -\lambda^\prime(k)j+y\}}\\-W(k)g_k(yj^{-1/2}) \int_0^\infty h_k(x) \dd x\Big| = 0 \quad \text{a.s.}
\end{multline*}
Further, the function $(k!)^{-1}-h_k$ is nonnegative and nondecreasing on $(-\infty, 0)$. In view of $(k!)^{-1}-h_k(x)\sim ((k-1)!(k+1))^{-1}e^x$ as $x\to-\infty$, it is Lebesgue integrable on $(-\infty, 0)$, hence dRi on $(-\infty, 0)$. By another appeal to Proposition~\ref{prop:bertoinExtended}, this time with $f(x)=((k!)^{-1}-h_k(x))\1_{(-\infty, 0)}(x)$,
\begin{multline*}
\lim_{j \to \infty}  \sup_{y \in \R} \Big|j^{1/2} \sum_{|u|=j}e^{-kV(u)-\lambda(k)j}((k!)^{-1}-h_k(-\lambda^\prime(k)j+y-V(u))\1_{\{V(u)>-\lambda^\prime(k)j+y\}}\\-W(k)g_k(yj^{-1/2})\int_{-\infty}^0 ((k!)^{-1}-h_k(x)){\rm d}x\Big| = 0 \quad \text{a.s.}
\end{multline*}
Finally, by Proposition~\ref{prop:cltBrw} with $f(x)=\1_{[0,\infty)}(x)$, for any $C>0$, a.s.
$$\lim_{j\to\infty}\sup_{|y|\leq Cj^{1/2}}\Big|\sum_{|u|=j}e^{-kV(u)-\lambda(k)j}\1_{\{V(u)>-\lambda^\prime(k)j+y\}}-W(k)\Phi\big(-y(j \lambda''(k))^{-1/2}\big)\Big|=0.$$
Combining fragments together we arrive at the claimed asymptotic relation for $\rmE \calK^{(j)}_{t^{(k)}_j(y)}(k)$.

For positive integer $\ell\leq k$, put $$\psi_{\ell,\, k}(x):=\phi_k(x)(1-\phi_\ell(x)),\quad x\geq 0\quad \text{and}\quad m_{\ell,\,k}(x):=e^{-kx} \psi_{\ell,\,k}(e^x) ,\quad x\in\R.$$ Passing to the analysis of the covariance (and the variance) we assume that there exist positive integer $\ell\leq k$ such that $\ell,k\in (\theta_\ast,\theta^\ast)$. We start with
\begin{multline*}
\rmCov\,(\calK^{(j)}_{t^{(k)}_j(y)}(\ell), \calK^{(j)}_{t^{(k)}_j(y)}(k))=\sum_{|u|=j}\psi_{\ell,\,k}(e^{-\lambda^\prime(k)j+y-V(u)})\\=e^{\varphi(k)j+y} \times\Big(\sum_{|u|=j}e^{-kV(u)-\lambda(k)j}m_{\ell,\,k}(-\lambda^\prime j+y-V(u))\1_{\{V(u)\leq -\lambda^\prime(k)j+y\}}\\-
\sum_{|u|=j}e^{-kV(u)-\lambda(k)j}\Big(\frac{1}{k!}-m_{\ell,\,k}(-\lambda^\prime(k)j+y-V(u))\Big)\1_{\{V(u)>-\lambda^\prime(k)j+y \}} \\+\frac{1}{k!}\sum_{|u|=j} e^{-kV(u)-\lambda(k)j}\1_{\{V(u)>-\lambda^\prime(k)j+y\}}\Big).
\end{multline*}
Note that $m_{\ell,\,k}(x)=h_k(x)(1-\phi_\ell(e^x))$ for $x\in\R$. We already know that $h_k$ is nonnegative and nonincreasing on $\R$. Since $\phi_\ell$ is the distribution function of a gamma distribution with parameters $\ell$ and $1$ (a.k.a.\ Erlang's distribution) we conclude that $x\mapsto 1-\phi_\ell(e^x)$ is nonnegative and nondecreasing on $\R$, too. Hence, $m_{\ell,\,k}$ is a nonincreasing function on $\R$. Since $$m_{\ell,\,k}(x)\sim  ((k-1)!)^{-1} \exp(-((k-\ell+1)x+e^x)),\quad x\to\infty,$$ it is Lebesgue integrable on $\R^+$, hence dRi on $\R^+$. By Proposition~\ref{prop:bertoinExtended} with $f(x)=m_{\ell,\,k}(x)\1_{[0,\infty)}(x)$,
\begin{multline*}
\lim_{j \to \infty}  \sup_{y \in \R} \Big|  j^{1/2} \sum_{|u|=j}e^{-kV(u)-\lambda(k)j}m_{\ell,\,k}(-\lambda^\prime(k)j+y-V(u))\1_{\{V(u)\leq -\lambda^\prime(k)j+y\}}\\-W(k)g_k(yj^{-1/2}) \int_0^\infty m_{\ell,\,k}(x) \dd x\Big| = 0 \quad \text{a.s.}
\end{multline*}
The function $(k!)^{-1}-m_{\ell,\,k}$ is nonnegative and nondecreasing on $(-\infty, 0)$. In view of $(k!)^{-1}-m_{\ell,\,k}(x)\sim c_{\ell,\,k} e^x$ as $x\to-\infty$, where $c_{1,\,k}=(2k+1)((k+1)!)^{-1}$ and $c_{\ell,\,k}=((k-1)!(k+1))^{-1}$ for $\ell\geq 2$, it is Lebesgue integrable on $(-\infty, 0)$, hence dRi on $(-\infty, 0)$. Invoking Proposition~\ref{prop:bertoinExtended} once again (with $f(x)=((k!)^{-1}-m_{\ell,\,k}(x))\1_{(-\infty, 0)}(x)$) we infer
\begin{multline*}
\lim_{j \to \infty}  \sup_{y \in \R} \Big|  j^{1/2} \sum_{|u|=j}e^{-kV(u)-\lambda(k)j}((k!)^{-1}-m_{\ell,\,k}(-\lambda^\prime(k)j+y-V(u)))\1_{\{V(u)\leq -\lambda^\prime(k)j+y\}}\\-W(k)g_k(yj^{-1/2}) \int_{-\infty}^0 ((k!)^{-1}-m_{\ell,\,k}(x))\dd x\Big| = 0 \quad \text{a.s.}
\end{multline*}
Finally, by Proposition~\ref{prop:cltBrw} with $f(x)=\1_{[0,\infty)}(x)$, for any $C>0$, a.s.
$$\lim_{j\to\infty}\sup_{|y|\leq Cj^{1/2}}\Big|\sum_{|u|=j}\1_{\{V(u)>-\lambda^\prime(k)j+y\}}e^{-kV(u)-\lambda(k)j}-W(k)\Phi\big(-y(j \lambda''(k))^{-1/2}\big)\Big|=0.$$
Combining fragments together we obtain the claimed formula for the covariance when $\ell<k$ and that for the variance when $\ell=k$.

Formula \eqref{eq:slln} is an immediate consequence of the already proved asymptotics of $\rmE \calK^{(j)}_{t^{(k)}_j(y)}(k)$ and $\rmVar \calK^{(j)}_{t^{(k)}_j(y)}(k)$, the Bienaymé-Chebyshev inequality and the Borel-Cantelli lemma in conjunction with the a.s.\ convergence of the series $\sum_{j\geq 1}(\rmE \calK^{(j)}_{t^{(k)}_j(y)}(k))^{-1}$.
\end{proof}

With Lemma~\ref{lem:belowSaturationPoissonk} at hand, we are ready to prove Theorem~\ref{thm:occupancyk}.
\begin{proof}[Proof of Theorem~\ref{thm:occupancyk}]
Let $b\in\R$ and $(j_n)_{n\in\N}$ be a sequence of positive integers satisfying, as $n\to\infty$, $j_n=(-\lambda^\prime(k))^{-1}\log n-b(\log n)^{1/2}+O((\log n)^{1/2})$.
For each $j \in \N$ and each $\ell\in\N$, the sequence $(K^{(j)}_n(\ell))_{n\in\N}$ is nondecreasing. Fix $\delta > 0$. Invoking the SLLN for Poisson processes and the aforementioned monotonicity we conclude that a.s.\ for $n$ large enough,
\begin{equation}\label{eq:4444}
\calK^{(j_n)}_{(1-\delta)n}(k) \leq K^{(j_n)}_n (k) \leq \calK^{(j_n)}_{(1+\delta)n}(k).
\end{equation}

Fix $\gamma>0$. We intend to apply formula \eqref{eq:slln} with $j=j_n=(-\lambda^\prime(k))^{-1}\log n+(-\lambda^\prime(k))^{-1}b(\log n)^{1/2}+o((\log n)^{1/2})$ and $y$ such that $t^{(k)}_{j_n}(y)=\gamma n$. The so defined $y=y_n$ satisfies $|y|\leq Cj_n^{1/2}$ for some $C>b$ and large $n$. Also, $\lim_{n\to\infty} (y_n j_n^{-1/2})=(-\lambda^\prime(k))^{1/2}b$. Hence, by \eqref{eq:slln}, as $n\to\infty$,
$$\calK^{(j)}_{\gamma n}(k)=\gamma (k!)^{-1}W(k) \Phi\big(-b((-\lambda^\prime(k))/\lambda''(k))^{1/2}\big)ne^{(\lambda(k)-(k-1)\lambda^\prime(k))j_n} (1+o(1))\quad {\rm a.s.}$$ This in combination with \eqref{eq:4444} proves Theorem~\ref{thm:occupancyk}.
\end{proof}

Part (III) of Theorem~\ref{thm:occupancy} is a particular case of Theorem~\ref{thm:occupancyk} with $k=1$. Note that $\lambda(1)=0$ and $W(1)=1$ a.s.

\subsection{Postsaturation levels. Proof of Theorem~\ref{thm:occupancy}(IV)} \label{subsec:aboveSaturation}

We now turn to the investigation of the number of occupied boxes $K_n^{(j)}$ in the postsaturation levels $j$. In this setting, the a.s.\ asymptotics of $K_n^{(j)}$ is driven, for the most part, by the number of `large' enough boxes. More precisely, assuming that $n$ balls are being thrown, there are $j$th level boxes of three types: (a) boxes with sizes (hitting probabilities) of order smaller than $1/n$ contain $0$ or $1$ ball with high probability; (b) boxes with sizes of order $1/n$ contain a Poisson number of balls with high probability; and (c) boxes with sizes of order larger than $1/n$. In the levels of low density, most boxes are of size $o(1/n)$, whence $K_n^{(j_n)}\approx n$. In the postsaturation levels, the sum of sizes of the aforementioned boxes is asymptotically very small, and most boxes are of large sizes. As a result, most balls fall into the same collection of boxes, and the number of occupied boxes is small with respect to $n$.

The results obtained in this section are an extension of Theorem 1 in \cite{Bertoin:2008}, in which the case $a j_n = \log n + b + o(1)$ is analyzed. We shall treat a more general case $a j_n = \log n + b (\log n)^{1/2}(1 + o(1))$, with $-\lambda'(1) < a < \bar{a}$.
\begin{lemma}\label{lem:poissonAboveSaturation}
Let $a\in (a_c, \bar a)$ and $t_j(y)= e^{aj+y}$ for $j\in\N$ and $y\in\R$. Pick $\theta \in (0,1)$ satisfying $a = -\lambda'(\theta)$.
The following asymptotic relations hold a.s.
\begin{equation*}
\lim_{j \to \infty} \sup_{y \in \R} \Big| j^{1/2} (t_j (y))^{-\theta} e^{-\lambda(\theta)j} \rmE \calK^{(j)}_{t_j(y)}-\frac{\Gamma(1-\theta)}{\theta} W(\theta) 
g_\theta(yj^{-1/2}) \Big| = 0 
\end{equation*}
\begin{equation*}
\lim_{j \to \infty} \sup_{y \in \R} \Big| j^{1/2} (t_j(y))^{-\theta} e^{-\lambda(\theta)j} \rmVar \calK^{(j)}_{t_j(y)}-\frac{(2^\theta-1)\Gamma(1-\theta)}{\theta} W(\theta)g_\theta(yj^{-1/2}) \Big| = 0 
\end{equation*}
\begin{equation}\label{eq:abovesat3}
\text{and} \quad \calK^{(j)}_{t_j(y)}= \frac{\Gamma(1-\theta)}{\theta}W(\theta)g_\theta(yj^{-1/2})j^{-1/2} (t_j (y))^\theta e^{\lambda(\theta)j}(1+o(1))
\end{equation}
as $j \to \infty$, uniformly in $y$ satisfying $|y|\leq Cj^{1/2}$, for any $C>0$. 
\end{lemma}
\begin{proof}
We start by noting that, for $j\in\N$ and $y\in\R$, $$\rmE \calK^{(j)}_{t_j(y)}= \sum_{|u|=j} \left(1 - e^{-t_j(y) e^{-V(u)}}\right).$$
We intend to apply Proposition~\ref{prop:bertoinExtended}. To this end, we use an alternative representation
\[
  \rmE \calK^{(j)}_{t_j(y)}(1)= e^{\lambda(\theta)j} e^{\theta (a j + y)}\sum_{|u|=j}e^{-\theta V(u) - \lambda(\theta)j}m(aj+y-V(u)),
\]
where $m(x):=(1 - e^{-e^x}) e^{-\theta x}$ for $x\in\R$. Integration by parts yields
\[
  \int_\R m(x) \dd x = \int_0^\infty x^{-\theta - 1}(1 - e^{-x}) \dd x =\frac{\Gamma(1-\theta)}{\theta},
\]
where the assumption $\theta \in (0,1)$ has to be recalled. Since the function $x\mapsto e^{\theta x}m(x)$ is increasing on $\R$, we conclude with the help of Lemma 9.1 in \cite{Goldie:1991} that $m$ is dRi on $\R$. Hence, by Proposition~\ref{prop:bertoinExtended} with $f=m$, the asymptotic formula for $\rmE \calK^{(j)}_{t_j(y)}$ holds.

The argument for the variance is similar. We start with
\begin{multline*}
\rmVar \calK^{(j)}_{t_j(y)}=\sum_{|u|=j} \left(1 - e^{-t_j(y) e^{-V(u)}}\right) e^{-t_j(y) e^{-V(u)}}\\=e^{\lambda(\theta)j} e^{\theta (a j + y)}\sum_{|u|=j}e^{-\theta V(u) -  \lambda(\theta)j}v(aj+y-V(u)),
\end{multline*}
where $v(x):= e^{-e^x}(1 - e^{-e^x}) e^{-\theta x}$ for $x\in\R$. Write $v=v_1-v_2$, where $v_1(x):= (1 - e^{-2e^x}) e^{-\theta x}$ and $v_2(x):=
(1 - e^{-e^x}) e^{-\theta x}$ for $x\in\R$. According to the previous part of the proof the functions $v_1$ and $v_2$ are dRi on $\R$. Furthermore,
\[\int_\R v(x) \dd x = \int_\R (v_1(x)-v_2(x))\dd x =\frac{(2^\theta-1)\Gamma(1-\theta)}{\theta}.
\]
Another application of Proposition~\ref{prop:bertoinExtended} to $f=v_1$ and then $f=v_2$ enables us to conclude that the asymptotic formula for $\rmVar \calK^{(j)}_{t_j(y)}$ holds.

SLLN \eqref{eq:abovesat3} follows by the same reasoning as given in the proof of Lemma~\ref{lem:belowSaturationPoissonk}. We only state explicitly that $\sum_{j\geq 1}( \rmE \calK^{(j)}_{t_j(y)})^{-1}<\infty$ a.s.\ which is the most important ingredient of the proof.
\end{proof}
We are ready to prove part (IV) of Theorem~\ref{thm:occupancy}.
\begin{proof}[Proof of Theorem~\ref{thm:occupancy}(IV)]
Let $a\in (a_c, \bar a)$ and $b\in\R$. Fix $\gamma>0$. We shall use relation \eqref{eq:abovesat3} with $j=j_n=a^{-1}\log n+a^{-1}b(\log n)^{1/2}+o((\log n)^{1/2})$ and $y=\log \gamma+b(\log n)^{1/2}-o((\log n)^{1/2})$. Then $t_{j_n}(y)=\gamma n$ and $\lim_{n\to\infty} (y_n j_n^{-1/2})=a^{1/2}b$. Hence, by \eqref{eq:abovesat3}, as $n\to\infty$, $$\calK^{(j)}_{\gamma n}(k)=\gamma^\theta \frac{\Gamma(1-\theta)}{\theta} W(\theta)  g_\theta(a^{1/2}b)j_n^{-1/2} n^\theta e^{\lambda(\theta)j_n} (1+o(1))\quad {\rm a.s.}$$ This in combination with the $k=1$ version of \eqref{eq:4444} proves part (IV) of Theorem~\ref{thm:occupancyk}.
\end{proof}

\subsection{Number of empty boxes at freezing} \label{subsec:freezing}

We assume in this section that $\lambda(0) < \infty$, that is, the mean number of boxes in the first level is finite. As a consequence, $\E Z_j=e^{-\lambda(0)j}$ for $j\in\N$. Furthermore,  
\[
W_j(0)= e^{-\lambda(0)j}Z_j~\to~W(0),\quad j\to\infty \quad \text{a.s.}
\]
If the number of balls thrown is $n = e^{a j + y}$ with $a > -\lambda(0)$, then most available boxes in the $j$th level will be occupied. Hence, to obtain a tight estimate on the number of occupied boxes in the $j$th level, one has to investigate the asymptotic behavior of $L_n^{(j)}$ the number of empty boxes in the $j$th level. Put $\mathcal{L}^{(j)}_t:=L^{(j)}_{N_t}$, where $(N_t)_{t\geq 0}$ is the same Poisson process as before.
\begin{lemma}\label{lem:free}
Let $a\in (-\lambda^\prime(0), \bar{a}_-)$ and $t_j(y) = e^{a j + y}$ for $j\in\N$ and $y\in\R$. Pick $\theta\in (\theta_\ast, 0)$ satisfying $a = -\lambda'(\theta)$.
Then
\[
  \lim_{j \to \infty} \sup_{y \in \R} \Big| j^{1/2} (t_j(y))^{-\theta} e^{-\lambda(\theta)j} \rmE \calL^{(j)}_{t_j(y)}- \Gamma(-\theta)W(\theta)g_\theta(yj^{-1/2})\Big|=0 \quad \text{{\rm a.s.}},
\]
\[
  \lim_{j \to \infty} \sup_{y \in \R} \Big|j^{1/2} (t_j(y))^{-\theta}e^{-\lambda(\theta)j} \rmVar \calL^{(j)}_{t_j(y)}- (1 - 2^\theta)\Gamma(-\theta) W(\theta)g_\theta(yj^{-1/2})\Big|=0  \quad \text{{\rm a.s.}}
\]
$$\text{and}\quad \calL^{(j)}_{t_j(y)}=\Gamma(-\theta)W(\theta)g_\theta(yj^{-1/2})j^{-1/2} (t_j(y))^{\theta}e^{\lambda(\theta)j}(1+o(1))\quad \text{{\rm a.s.}}$$ as $j \to \infty,$ uniformly in $y$ satisfying $|y|\leq Cj^{1/2}$, for any $C>0$. 
\end{lemma}
\begin{proof}
Observe that, for $j \in \N$ and $y \in \R$,
\begin{multline*}
\rmE \calL^{(j)}_{t_j(y)}= \sum_{|u|=j} \exp\left( - e^{a j+y - V(u)} \right) \quad \text{and} \\
  \rmVar \calL^{(j)}_{t_j(y)}= \sum_{|u|=j} \exp\left( - e^{a j +y - V(u)} \right)\left( 1 - \exp\left(- e^{aj+y-V(u)}\right) \right).
\end{multline*}
We represent the mean in an equivalent form
\[
  \rmE \calL^{(j)}_{t_j(y)}= e^{\theta(aj+y)+ \lambda(\theta)j}\sum_{|u|=j} e^{-\theta V(u) - \lambda(\theta)j} m(aj+y-V(u)),
\]
where $m(x) = e^{-e^x} e^{-\theta x}$ for $x\in\R$. The function $x\mapsto e^{\theta x}m(x)$ is decreasing on $\R$ and \[\int_\R m(x) \dd x = \int_\R e^{-e^x} e^{-\theta x} \dd x = \int_0^\infty x^{-\theta -1} e^{-x} \dd x = \Gamma(-\theta)<\infty.\]
Recalling that $\theta<0$ we conclude that $m$ is dRi on $\R$ by Lemma 9.1 in \cite{Goldie:1991}. Using Proposition~\ref{prop:bertoinExtended} with $f=m$ yields
\[
  \lim_{j \to \infty} \sup_{y \in \R} \Big|j^{1/2}(t_j(y))^{-\theta}e^{-\lambda(\theta)j} \rmE \calL^{(j)}_{t_j(y)}- W(\theta) g_\theta(yj^{-1/2}) \int_\R m(x) \dd x \Big| = 0 \quad \text{a.s.},
\]
thereby proving the claim concerning the (conditional) mean. As far as the (quenched) variance is concerned, we write
\[
  \rmVar \calL^{(j)}_{t_j(y)}= e^{\theta(aj+y)+ \lambda(\theta)j}\sum_{|u|=j} e^{-\theta V(u) - \lambda(\theta)j}v(aj+y-V(u)),
\]
where $v(x) = e^{-e^x}(1-e^{-e^x}) e^{-\theta x}$ for $x\in\R$. The function $v$ is dRi on $\R$ as the difference of two dRi functions and
$$\int_\R v(x){\rm d}x=(1-2^\theta)\Gamma(-\theta)<\infty.$$ Hence, by Proposition~\ref{prop:bertoinExtended} with $f=v$,
\[
  \lim_{j \to \infty} \sup_{y \in \R} \Big|j^{1/2} (t_j(y))^{-\theta} e^{-\lambda(\theta)j} \rmVar \calL^{(j)}_{t_j(y)}- W(\theta) g_\theta(yj^{-1/2}) \int_\R v(x)\dd x \Big| =0 \quad \text{a.s.}
\]

The SLLN for $\calL^{(j)}_{t_j(y)}$ is secured by the already proved uniform asymptotic estimates for the (conditional) mean and variance of $\calL^{(j)}_{t_j(y)}$. See the proof of Lemma~\ref{lem:belowSaturationPoisson} for details.
\end{proof}

Theorem~\ref{thm:occupancyFreezing} will now be proved with the help of Lemma~\ref{lem:free}.
\begin{proof}[Proof of Theorem~\ref{thm:occupancyFreezing}]
Let $a\in (-\lambda^\prime(0), \bar{a}_-)$ and $b\in\R$. Fix $\delta > 0$ and note that, for each $j \in \N$, the sequence $(L^{(j)}_n)_{n\in\N}$ is a.s.\ nonincreasing. By the SLLN for Poisson processes, $N_{(1-\delta)n}\leq n\leq N_{(1+\delta)n}$ a.s.\ for $n$ large enough. Hence, a.s.\ for $n$ large enough,
\begin{equation}\label{eq:128}
\calL^{(j_n)}_{(1+\delta)n} \leq L^{(j_n)}_n \leq \calL^{(j_n)}_{(1-\delta)n}.
\end{equation}
Fix any $\gamma>0$. We shall apply the SLLN for $\calL^{(j)}_{t_j(y)}$ from Lemma~\ref{lem:free} with $j=j_n=a^{-1}\log n+a^{-1}b(\log n)^{1/2}+o((\log n)^{1/2})$ and $y=\log \gamma+b(\log n)^{1/2}-o((\log n)^{1/2})$. The so defined $y=y_n$ satisfies $|y|\leq Cj_n^{1/2}$ for some $C>b$ and large $n$ and $t_{j_n}(y)=\gamma n$. Also, $\lim_{n\to\infty} (y_n j_n^{-1/2})=a^{1/2}b$. Hence, as $n\to\infty$, $$\calL^{(j_n)}_{\gamma n}=\gamma^\theta \Gamma(-\theta)W(\theta)g_\theta(a^{1/2}b) j_n^{-1/2}n^\theta e^{\lambda(\theta)j_n}(1+o(1))\quad \text{a.s.}$$ This in combination with \eqref{eq:128} completes the proof of the theorem.
\end{proof}

\appendix

\section{Appendix: proofs of Propositions~\ref{prop:bertoinExtended} and~\ref{prop:cltBrw} and Corollary~\ref{lem:aux22}}
\label{app:proofProps}

We prove in this section a uniform central limit theorem and a uniform local limit theorem for the sequence of random measures $(Z_\theta^{(j)})_{j\in\N}$ defined, for $\theta\in (\theta_\ast, \theta^\ast)$, by 
\[Z_\theta^{(j)}=\sum_{|u|=j} e^{-\theta V(u) - \lambda(\theta)j}\delta_{V(u)},\quad j\in\N.\]

\subsection{Proof of Proposition~\ref{prop:bertoinExtended}}

For $j \in \N$, $\theta\in (\theta_\ast,\theta^\ast)$ and $y \in \R$, we write $\beta_j(\theta, y):=-\lambda^\prime(\theta)j+y$, so that
\[
  j^{1/2}\sum_{|u|=j}e^{-\theta V(u)-\lambda(\theta) j} f(-\lambda^\prime(\theta)j+y-V(u)) = j^{1/2}\int_{\R^+}f(\beta_j(\theta, y)-x)Z^{(j)}_\theta({\rm d}x).
\]
Our argument is similar to the proof of the key renewal theorem as given on pp.~241--242 in \cite{Resnick:2002}. We proceed via three steps, complicating successively the structure of $f$.

\noindent {\sc Step 1}. Suppose first that, for fixed integer $n$ and fixed $h>0$,
$$f(t)=\1_{[(n-1)h,\,nh)}(t),\quad t\in\R.$$ Then $f(\beta_j(\theta, y)-x)=1$ if, and only if, $x\in
(\beta_j(\theta, y)-nh,\,\beta_j(\theta, y)-(n-1)h]$ which entails, for large $j$, $$j^{1/2}\int_{\R^+}f(\beta_j(\theta, b)-x)Z^{(j)}_\theta({\rm d}x)=
j^{1/2}Z^{(j)}_\theta ((\beta_j(\theta, y)-nh,\, \beta_j(\theta, y)-(n-1)h]).$$ Replacing in \eqref{eqn:bigg1} $h$ with $h/2$, then putting $x=y-nh+h/2$ and exploiting the uniform continuity of $g_\theta$ we conclude that, according to Lemma~\ref{fct:biggStep},
$$\lim_{j \to \infty} \sup_{y \in \R} \Big|j^{1/2}\int_{\R^+}f(\beta_j(\theta, y)-x)Z^{(j)}_\theta({\rm d}x) - W(\theta)g_\theta(j^{-1/2}y)h \Big| = 0 \quad \text{a.s.}$$ Since $\int_\R f(x) \dd x= h$, we have proved that $$\lim_{j \to \infty} \sup_{y \in \R} \left|j^{1/2}\int_{\R^+}f(\beta_j(\theta, y)-x)Z^{(j)}_\theta({\rm d}x) - W(\theta)g_\theta(j^{-1/2}y) \int_\R f(x) \dd x \right| = 0 \quad \text{a.s.}$$

\noindent {\sc Step 2}. Suppose now that $$f(t)=\sum_{n\in\Z} c_n\1_{[(n-1)h,\,nh)}(t),\quad t\in\R,$$ where
$(c_n)_{n\in\Z}$ is a sequence of nonnegative numbers satisfying $\sum_{n\in\Z}c_n<\infty$. Then
\begin{multline*}
j^{1/2}\int_{\R^+}f(\beta_j(\theta,y)-x)Z_\theta^{(j)}({\rm d}x)\\
=j^{1/2}\sum_{n\leq \beta_j(\theta, y)/h} c_n Z^{(j)}_\theta ((\beta_j(\theta, y)-nh,\, \beta_j(\theta, y)-(n-1)h]).
\end{multline*}
Combining the conclusion of Step 1 and the triangular inequality we obtain, for each $m \in \N$,
\begin{multline*}
\lim_{j \to \infty} \sup_{y \in \R} \Big|j^{1/2}\sum_{n=-m}^m c_n Z^{(j)}_\theta ((\beta_j(\theta, y)-nh,\, \beta_j(\theta, y)-(n-1)h])\\
- W(\theta) g_\theta(j^{-1/2}y)\int_{-(m+1)h}^{mh}f(x) \dd x\Big| = 0 \quad \text{a.s.}
\end{multline*}
Given $\varepsilon>0$ choose $m\in \N$ so large that $\sum_{|n|>m} c_n \leq \epsilon$. Using this together with the previous limit relation yields
\begin{multline}
\limsup_{j \to \infty} \sup_{y \in \R} \Big| j^{1/2}\sum_{n=-m}^m c_n Z^{(j)}_\theta ((\beta_j(\theta, y)-nh,\, \beta_j(\theta, y)-(n-1)h])\\
- W(\theta) g_\theta(j^{-1/2}y)\int_\R f(x) \dd x \Big|\leq \epsilon W(\theta) g_\theta(0) \quad \text{a.s.}\label{eq:2}
\end{multline}
Invoking Lemma~\ref{fct:biggStep} with the same $h$ and $x$ as at Step 1 we conclude that a.s.\ there exists $j_0(\epsilon)\in\N$ such that, for all $j \geq j_0(\epsilon)$ and $n \in \N$,
$$j^{1/2} Z^{(j)}_\theta ((\beta_j(\theta, y)-nh,\, \beta_j(\theta, y)-(n-1)h]) \leq h W(\theta) g_\theta(0) + \epsilon.$$ As a consequence,
\begin{multline}
\limsup_{j \to \infty} \sup_{y \in \R}\sum_{|n|>m}c_n Z^{(j)}_\theta ((\beta_j(\theta, y)-nh,\, \beta_j(\theta, y)-(n-1)h])\\\leq \sum_{|n|>m}c_n \left(h W(\theta) g_\theta(0) + \epsilon\right)
\leq \epsilon (h W(\theta) g_\theta(0) + \epsilon)\quad\text{a.s.}\label{eq:3}
\end{multline}
Letting in \eqref{eq:2} and \eqref{eq:3} $\epsilon \to 0+$ we arrive at \eqref{eq:exten}, with $f$ considered at this step.

\noindent {\sc Step 3}.
Let now $f$ be an arbitrary nonnegative dRi function on $\R$. For each $h>0$, put
$$\overline{f}_h(t):=\sum_{n\in\Z}\underset{(n-1)h\leq y<nh}{\sup}\,f(y)\1_{[(n-1)h,\,nh)}(t),\quad t\in\R$$
and
$$\underline{f}_h(t):=\sum_{n\in\Z}\underset{(n-1)h\leq y<nh} {\inf}\,f(y)\1_{[(n-1)h,\,nh)}(t),\quad t\in\R.$$
By the definition of direct Riemann integrability,
$$\sum_{n\in\Z} \underset{(n-1)h\leq y<nh}{\sup}\,f(y)<\infty \quad \text{and}\quad \sum_{n\in\Z}\underset{(n-1)h\leq y<nh}{\inf}\,f(y)<\infty$$
for each $h>0$. Thus, the functions $\overline{f}_h$ and $\underline{f}_h$ have the same structure as the functions discussed in Step 2. According to the result of Step 2, for $h>0$,
\[
  \lim_{j \to \infty} \sup_{y \in \R} \left|\int_{\R^+} \overline{f}_h(\beta_j(\theta,y)-x)Z_\theta^{(j)}({\rm d}x) -  W(\theta) g_\theta(yj^{-1/2}) \int_\R \overline{f}_h(x) \dd x \right| = 0\quad \text{a.s.}
\]
and
\[
  \lim_{j \to \infty} \sup_{y \in \R} \left|\int_{\R^+} \underline{f}_h(\beta_j(\theta,y)-x)Z_\theta^{(j)}({\rm d}x) -  W(\theta) g_\theta(yj^{-1/2}) \int_\R \underline{f}_h(x) \dd x \right| = 0\quad\text{a.s.}
\]
Now \eqref{eq:exten} follows from these limit relations and the inequality
\begin{multline*}
   \int_{\R^+} \underline{f}(\beta_j(\theta,y)-x) Z_\theta^{(j)}(\dd x) \\
   \leq \int_{\R^+} f(\beta_j(\theta,y)-x) Z_\theta^{(j)}(\dd x) \leq \int_{\R^+} \overline{f}(\beta_j(\theta,y)-x)Z_\theta^{(j)}(\dd x),
\end{multline*}
after noting that $$\int_\R  \overline{f}_h(x){\rm d}x=h\sum_{n\in\Z}\underset{(n-1)h\leq y<nh}{\sup}\,f(y)~\to~\int_\R f(x) \dd x,\quad h\to 0+$$ and
$$\int_\R \underline{f}_h(x){\rm d}x=h\sum_{n\in\Z}\underset{(n-1)h\leq y<nh}{\inf}\,f(y)~\to~\int_\R f(x) \dd x,\quad h\to 0+.$$

\subsection{Proof of Proposition~\ref{prop:cltBrw}}

The proof will be divided into three steps. First, we treat indicator functions, second, Riemann integrable functions of compact supports and third, general bounded, a.e.\ continuous functions.

\noindent {\sc Step 1}. Assume that $f(t)=\1_{[a,b)}(t)$, $t\in\R$ for fixed $a,b\in \R$, $a<b$. We intend to show that
\begin{multline}\label{eq:inter1}
\lim_{j\to\infty} \sum_{|u|=j}e^{-\theta V(u)-\lambda(\theta) j}f\Big(\frac{-\lambda^\prime(\theta)j-V(u)}{j^{1/2}}+ y\Big)=W(\theta) \int_0^{b-a}g_\theta(y-a-x){\rm d}x\\= W(\theta) \int_\R f(y-x) g_\theta(x) \dd x\quad \text{a.s.}
\end{multline}
uniformly in $y\in\R$.

Note that $$\sum_{|u|=j}e^{-\theta V(u)-\lambda(\theta) j} f\Big(\frac{-\lambda^\prime(\theta)j-V(u)}{j^{1/2}}+ y\Big)= \int_\R f\Big(\frac{-\lambda^\prime(\theta)j-x}{j^{1/2}}+y\Big) Z_\theta^{(j)}(\dd x).$$ We shall work with the integral on the right-hand side of this equality and start with a representation
\begin{multline*}
\int_\R f\Big(\frac{-\lambda^\prime (\theta)j-x}{j^{1/2}}+y \Big) Z_\theta^{(j)}(\dd x) \\
= \sum_{n=1}^\floor{(b-a)j^{1/2}} Z_\theta^{(j)}\Big((-\lambda'(\theta)j + y j^{1/2}- t_n^{(j)}, - \lambda^\prime(\theta)j + yj^{1/2}-t_{n-1}^{(j)}] \Big),
\end{multline*}
where $t_n^{(j)}:= aj^{1/2} + n\frac{(b-a)j^{1/2}}{\floor{(b-a)j^{1/2}}}$ for $n\in\N_0$. Using Lemma \ref{fct:biggStep} with $x=yj^{1/2}-(t_n^{(j)}+t_{n-1}^{(j)})/2$ and $h=(t_n^{(j)}-t_{n-1}^{(j)})/2$ in combination with uniform continuity of $g_\theta$ and noting that $$2h=t_n^{(j)}-t_{n-1}^{(j)}=\frac{(b-a)j^{1/2}}{\floor{(b-a)j^{1/2}}}~\to~1,\quad j\to\infty$$ we infer 
\begin{multline*}
\lim_{j \to \infty} \sup_{y \in \R} \Big| j^{1/2}Z_\theta^{(j)}\Big( (- \lambda^\prime(\theta)j + y j^{1/2}-t_n^{(j)}, - \lambda^\prime(\theta)j + yj^{1/2}-t_{n-1}^{(j)}] \Big)\\
   - W(\theta) g_\theta(y-t^{(j)}_n/j^{1/2})\Big| = 0 \quad \text{a.s.}
\end{multline*}
Therefore, given $\epsilon > 0$ there exists an a.s.\ finite $j_0$ such that, for all $j\geq j_0$ and all $y \in \R$,
\begin{multline}\label{eq:inter2}
\int_\R f\Big( \frac{-\lambda'(\theta)j-x}{j^{1/2}}+y\Big) Z_\theta^{(j)}(\dd x) \leq \epsilon (b-a)\\ + W(\theta)j^{-1/2}\sum_{n=1}^{\lfloor (b-a)j^{1/2}\rfloor} g_\theta\Big(y-a-\frac{n(b-a)}{\lfloor (b-a)j^{1/2}\rfloor}\Big).
\end{multline}
Plainly, by the definition of the Riemann integral, $$\lim_{j\to\infty} j^{-1/2}\sum_{n=1}^{\lfloor (b-a)j^{1/2}\rfloor} g_\theta\Big(y-a-\frac{n(b-a)}{\lfloor (b-a)j^{1/2}\rfloor}\Big)=
\int_0^{b-a}g_\theta(y-a-x){\rm d}x.$$ However, we still have to prove the uniformity. To this end, we exploit uniform continuity of $g_\theta$. Putting, for $\delta>0$,
$\omega(\delta):=\sup_{x,y\in\R, |x-y|\leq \delta}\,|g_\theta(x)-g_\theta(y)|$, we write
\begin{multline*}
\sup_{y\in\R}\,\Big|\int_0^{b-a}g_\theta(y-a-x){\rm d}x-j^{-1/2}\sum_{n=1}^{\lfloor (b-a)j^{1/2}\rfloor} g_\theta\Big(y-a-\frac{n(b-a)}{\lfloor (b-a)j^{1/2}\rfloor}\Big)\Big|\\=
\sup_{y\in\R}\,\Big|\sum_{n=1}^{\lfloor (b-a)j^{1/2}\rfloor}\Big(\int_{\frac{(n-1)(b-a)}{\lfloor (b-a)j^{1/2}\rfloor}}^{\frac{n(b-a)}{\lfloor (b-a)j^{1/2}\rfloor}}g_\theta(y-a-x){\rm d}x-j^{-1/2} g_\theta\Big(y-a-\frac{n(b-a)}{\lfloor (b-a)j^{1/2}\rfloor}\Big)\Big)\Big|\\=\sup_{y\in\R}\,\Big|\sum_{n=1}^{\lfloor (b-a)j^{1/2}\rfloor}\Big(\int_{\frac{(n-1)(b-a)}{\lfloor (b-a)j^{1/2}\rfloor}}^{\frac{n(b-a)}{\lfloor (b-a)j^{1/2}\rfloor}}\Big(g_\theta(y-a-x)-g_\theta\Big(y-a-\frac{n(b-a)}{\lfloor (b-a)j^{1/2}\rfloor}\Big)\Big){\rm d}x\\+\Big(\frac{b-a}{\lfloor (b-a)j^{1/2}\rfloor}-\frac{1}{j^{1/2}}\Big)g_\theta\Big(y-a-\frac{n(b-a)}{\lfloor (b-a)j^{1/2}\rfloor} \Big)\Big)\Big|\\ \leq (b-a)\omega\Big(\frac{b-a}{\lfloor (b-a)j^{1/2}\rfloor}\Big)+\Big(\frac{b-a}{\lfloor (b-a)j^{1/2}\rfloor}-\frac{1}{j^{1/2}}\Big)\sup_{x\in\R}\,g_\theta(x)~\to~0,\quad j\to\infty.
\end{multline*}
This in combination with \eqref{eq:inter2} yields
$${\lim\sup}_{j\to\infty} \int_\R f\Big( \frac{-\lambda'(\theta)j-x}{j^{1/2}}+y\Big) Z_\theta^{(j)}(\dd x)\leq W(\theta)\int_0^{b-a}g_\theta(y-a-x){\rm d}x\quad\text{a.s.}$$ uniformly in $y\in\R$. An analogous argument proves the converse inequality for the limit inferior,
whence \eqref{eq:inter1}.

\noindent {\sc Step 2}. Assume that $f$ is a Riemann integrable function on $[a,\,b]$ which is equal to $0$ outside $[a,\,b]$. We claim that relation \eqref{eq:inter1} also holds for such $f$.

To prove this, we note that given $\varepsilon>0$ there exist $n\in\N$ and a partition $a=x_0<x_1<\ldots<x_n=b$ such that
\begin{equation}\label{eq:inter3}
\int_a^b (\overline{f}_n(x)-f(x)){\rm d}x\leq \varepsilon\quad \text{and}\quad \int_a^b (f(x)-\underline{f}_n(x)){\rm d}x\leq \varepsilon,
\end{equation}
where, for $x\in [a,\,b)$, $$\overline{f}_n(x):=\sum_{i=1}^n \max_{x_{i-1}\leq z\leq x_i}\,f(z)\1_{[x_{i-1},\,x_i)}(x) \text{ and } \underline{f}_n(x):=\sum_{i=1}^n \min_{x_{i-1}\leq z\leq x_i}\,f(z)\1_{[x_{i-1},\,x_i)}(x).$$ According to the result of Step 1, uniformly in $y\in\R$,
\begin{equation*}
\lim_{j \to \infty} \Big|\sum_{|u|=j}e^{-\theta V(u) - \lambda(\theta)j}  \underline{f}_n(\tfrac{-\lambda^\prime(\theta)j-V(u)}{j^{1/2}}+ y)- W(\theta) \int_\R \underline{f}_n(x) g_\theta(y-x) \dd x \Big| = 0\quad\text{a.s.}
\end{equation*}
and
\begin{equation*}
\lim_{j \to \infty} \Big| \sum_{|u|=j}e^{-\theta V(u) - \lambda(\theta)j}\overline{f}_n(\tfrac{-j \lambda^\prime(\theta)-V(u)}{j^{1/2}} - y) - W(\theta) \int_\R \overline{f}_n(x) g_\theta(y-x) \dd x \Big| = 0 \quad \text{a.s.}
\end{equation*}
Using
\begin{multline*}
\sum_{|u|=j}e^{-\theta V(u) - \lambda(\theta)j}\underline{f}_n(\tfrac{-\lambda'(\theta)j-V(u)}{j^{1/2}}+ y) \leq \sum_{|u|=j}e^{-\theta V(u) -\lambda(\theta)j} f(\tfrac{-\lambda'(\theta)j-V(u)}{j^{1/2}}+y)\\
\leq \sum_{|u|=j}e^{-\theta V(u) -  \lambda(\theta)j}\overline{f}_n(\tfrac{-\lambda'(\theta)j-V(u)}{j^{1/2}}+ y),
\end{multline*}
we infer
\begin{multline}\label{eq:inter4}
\Big|\sum_{|u|=j}e^{-\theta V(u) - \lambda(\theta)j}f(\tfrac{-\lambda^\prime(\theta)j-V(u)}{j^{1/2}}+ y) - W(\theta) \int_\R f(x) g_\theta(y-x) \dd x\Big|\\\leq
\Big| \sum_{|u|=j}e^{-\theta V(u) - \lambda(\theta)j}\overline{f}_n(\tfrac{-j \lambda^\prime(\theta)-V(u)}{j^{1/2}} - y) - W(\theta) \int_\R \overline{f}_n(x) g_\theta(y-x) \dd x \Big|\\+
\Big| \sum_{|u|=j}e^{-\theta V(u) - \lambda(\theta)j}\underline{f}_n(\tfrac{-j \lambda^\prime(\theta)-V(u)}{j^{1/2}} - y) - W(\theta) \int_\R \underline{f}_n(x) g_\theta(y-x) \dd x \Big|\\+
\int_\R (\overline{f}_n(x)-f(x)) g_\theta(y-x) \dd x+ \int_\R (f(x)-\overline{f}_n(x)) g_\theta(y-x) \dd x.
\end{multline}
Observe that \eqref{eq:inter3} entails $$\int_\R (\overline{f}_n(x)-f(x)) g_\theta(y-x) \dd x\leq g_\theta(0)\varepsilon\quad \text{and}\quad
\int_\R (f(x)-\overline{f}_n(x)) g_\theta(y-x) \dd x\leq g_\theta(0)\varepsilon.$$ With this at hand, letting in \eqref{eq:inter4} $j\to\infty$ and then $\varepsilon\to 0+$ we arrive at \eqref{eq:inter1} with the present $f$.

\noindent {\sc Step 3}. To treat arbitrary bounded, a.e.\ continuous functions $f$, we fix $A>0$, $C\in (0,A)$ and put, for $x \in \R$,
$$f_A(x):=f(x)\1_{[-A,\,A]}(x)\quad \text{and}\quad \bar f_A(x):=f(x)-f_A(x)=f(x)\1_{\R\backslash [-A,\,A]}(x).$$ According to the result of Step 2,
\[
\lim_{j \to \infty}\sup_{y\in\R}\,
\Big| \sum_{|u|=j}e^{-\theta V(u) - \lambda(\theta)j} f_A(\tfrac{-\lambda^\prime(\theta)j-V(u)}{j^{1/2}} - y)- W(\theta) \int_\R f_A(y-x) g_\theta(x) \dd x \Big|=0 \quad \text{a.s.}
\]
Since $$\int_\R f_A(y-x) g_\theta(x) \dd x=\int_{y-A}^{y+A} f(y-x) g_\theta(x) \dd x~\to~\int_\R f(y-x) g_\theta(x) \dd x,\quad A\to\infty,$$ and the last integral is convergent, it suffices to show that
\begin{equation}\label{eq:aux}
\lim_{A\to\infty}{\lim\sup}_{j\to\infty}\sup_{|y|\leq C} \Big| \sum_{|u|=j}e^{-\theta V(u) - \lambda(\theta)j}\bar f_A(\tfrac{-\lambda^\prime(\theta)j-V(u)}{j^{1/2}} - y)\Big|=0\quad\text{a.s.}
\end{equation}
Observe that, for $y\in [-C,\,C]$,
\begin{multline}\label{eq:aux1}
|\bar f_A(\tfrac{-\lambda^\prime(\theta)j-V(u)}{j^{1/2}} - y)|\leq \sup_{x\in\R}|f(x)|\1_{\{|-\lambda^\prime(\theta)j-V(u)-yj^{1/2}|>Aj^{1/2}\}}\\\leq
\sup_{x\in\R}|f(x)|\1_{\{|-\lambda^\prime(\theta)j-V(u)|>(A-C)j^{1/2}\}}.
\end{multline}
Further, since $(W_j(\theta), \mathcal{F}_j)_{j\in\N_0}$ is a nonnegative martingale,
\[
\sum_{|u|=j} e^{-\theta V(u) - j \lambda(\theta)}=W_j(\theta)~\to~ W(\theta),\quad j\to\infty~~\text{a.s.}
\]
Also, according to the result of Step 1,
\[\lim_{j \to \infty} \sum_{|u|=j}\1_{\{|-\lambda^\prime(\theta)j-V(u)|\leq (A-C) j^{1/2}\}}= W(\theta) \P\{|N_\theta|\leq A-C\}\quad \text{a.s.},
\]
where $N_\theta$ is a random variable with density $g_\theta$. Hence, \[\lim_{j \to \infty} \sum_{|u|=j}e^{-\theta V(u) - \lambda(\theta)j}\1_{\{|-\lambda'(\theta)j-V(u)|>(A-C) j^{1/2}\}}= W(\theta) \P\{|N_\theta|>A-C\}\quad \text{a.s.}\] This in combination with \eqref{eq:aux1} proves \eqref{eq:aux}.

\subsection{Proof of Corollary~\ref{lem:aux22}}
On the one hand,
\begin{multline*}
\sum_{|u|=j}e^{-\vartheta V(u)-\lambda(\vartheta) j} \1_{\{V(u)\geq \delta j+yj^{1/2}\}} \leq \sum_{|u|=j}e^{-\vartheta V(u)-\lambda(\vartheta) j}\\= W_j(\vartheta)~\to~W(\vartheta),\quad j\to\infty\quad \text{a.s.}
\end{multline*}
uniformly in $y\in\R$. On the other hand, given $A>0$, for $j\geq (A/(-\lambda^\prime(\vartheta)-\delta))^2$, we infer
\begin{multline*}
\sum_{|u|=j}e^{-\vartheta V(u)-\lambda(\vartheta) j}\1_{\{V(u)\geq \delta j+yj^{1/2}\}}\geq \sum_{|u|=j}e^{-\vartheta V(u)-\lambda(\vartheta) j}\1_{\{V(u)\geq -\lambda^\prime(\vartheta)j+(y-A)j^{1/2}\}}\\~\to~ W(\vartheta)\int_{y-A}^\infty g_\vartheta(x){\rm d}x,\quad j\to\infty\quad \text{a.s.}
\end{multline*}
uniformly in $|y|\leq C$ for any $C>0$. The last limit relation is ensured by Proposition~\ref{prop:cltBrw} with $f(x)=\1_{(-\infty,\, 0]}(x)$, $x\in\R$. It remains to send $A$ to $\infty$.

\paragraph*{Acknowledgements.} B.M. is partially supported by the ANR grant MALIN (ANR-16-CE93-0003) and by LABEX MME-DII.

\end{document}